\documentclass[11pt,reqno]{amsart}
\usepackage{etex,amsfonts,amssymb,amsmath,amscd,
euscript,array,mathrsfs,comment}
\usepackage[all]{xy}

\numberwithin{equation}{section}
\newcommand{\dd}{\mathsf{d}}

\newcommand{\ldd}{\mathrm L}
\newcommand{\dddd}{\mathsf D}
\newcommand{\lddd}{L}
\newcommand{\Funct}{\Man^\Gr}
\newcommand{\Gr}{\mathsf{Gr}}
\newcommand{\Man}{\mathsf{Man}}
\newcommand{\Natu}{\mathsf{Nat}}
\newcommand{\Sman}{\mathsf{SMan}}
\newcommand{\fphi}{f}
\newcommand{\fpsi}{g}
\newcommand{\Tan}{\mathsf T}
\newcommand{\yek}{\mathbf 1}
\newcommand{\hck}{\check{h}^{\langle g\rangle}}
\newcommand{\htil}{h^{g_0}}
\newcommand{\bfc}{\mathbf f}
\newcommand{\frA}{\mathbf A}
\newcommand{\frB}{\mathbf B}

\newcommand{\frD}{\mathbf D}
\newcommand{\bcsfr}{\mathbf c_0}
\newcommand{\bcye}{\mathbf c_1}
\newcommand{\bcdo}{\mathbf c_2}

\newcommand{\EEsu}{\mathcal E^E}
\newcommand{\EFsu}{\mathcal E^F}
\newcommand{\Usu}{\mathcal U}
\newcommand{\Vsu}{\mathcal V}
\newcommand{\Wsu}{\mathcal W}
\newcommand{\Gsu}{\mathcal G}
\newcommand{\Fsu}{\mathcal F}
\newcommand{\Hsu}{\mathcal H}
\newcommand{\Msu}{\mathcal M}
\newcommand{\Nsu}{\mathcal N}
\newcommand{\Egsu}{\mathcal E^{\g g}}
\newcommand{\Vgzu}{\mathcal U^{\langle g_0\rangle }}
\newcommand{\kgz}{k^{{\langle g_0\rangle}}}

\newcommand{\g}{\mathfrak}

\renewcommand{\hom}{\mathsf{Mor}}
\newcommand{\homrm}{\mathrm{Hom}}
\newcommand{\higuc}{C^\infty(G,\g g)}
\newcommand{\higuo}{C^\omega(G,\g g)}

\newcommand{\ood}{{\overline{1}}\,}
\newcommand{\eev}{{\overline{0}}\,}
\newcommand{\End}{\mathrm{End}}
\newcommand{\Lie}{\mathrm{Lie}}
\newcommand{\Ad}{\mathrm{Ad}}

\newcommand{\N}{\mathbb N}
\newcommand{\R}{\mathbb R}
\newcommand{\C}{\mathbb C}

\newcommand{\Z}{\mathbb Z}
\newcommand{\eps}{\varepsilon}

\newcommand{\zro}{\mathbf 0}
\newcommand{\Smi}{\mathsf S}

\newtheorem{theorem}{\textbf{Theorem}}[section]
\newtheorem{deff}[theorem]{\textbf{Definition}}
\newtheorem{proposition}[theorem]{\textbf{Proposition}}
\newtheorem{lemma}[theorem]{\textbf{Lemma}}
\newtheorem*{notation}{\textbf{Notation}}

\newcommand{\FunC}{\mathcal E^{\C^{1|1}}}
\newcommand{\Sg}{\breve}

\newenvironment{definition}{\begin{deff}\rmfamily\upshape}{\end{deff}}
\newtheorem{corollary}[theorem]{\textbf{Corollary}}
\newtheorem{rmk}[theorem]{\textbf{Remark}}
\newtheorem{examp}[theorem]{\textbf{Example}}
\newenvironment{remark}{\begin{rmk}\rmfamily\upshape}{\end{rmk}}
\newenvironment{example}{\begin{examp}\rmfamily\upshape}{\end{examp}}

\newcommand{\fg}{\mathfrak g}

\newcommand{\res}{\Big|_}

\DeclareFontFamily{OT1}{pzc}{}
\DeclareFontShape{OT1}{pzc}{m}{it}{<-> s * [1.100] pzcmi7t}{}
\DeclareMathAlphabet{\mathpzc}{OT1}{pzc}{m}{it}

\title[
Positive definite superfunctions and 
unitary representations
]{
Positive definite superfunctions and 
unitary representations of Lie supergroups
}

\author{Karl--Hermann Neeb, Hadi Salmasian}

\thanks{H. Salmasian was supported by an NSERC Discovery Grant and  the Emerging Field Program at Universit\" at Erlangen--N\"urnberg.}

\address{
Department Mathematik\\
FAU Erlangen-N\"urnberg\\
Cauerstra\ss e 11, 91058 Erlangen, Deutschland 
}
\email{karl-hermann.neeb@math.uni-erlangen.de}

\address{Department of Mathematics and Statistics\\University of Ottawa\\ 585 King Edward Ave.\\ Ottawa, ON K1N 6N5\\ Canada}

\email{hsalmasi@uottawa.ca}


\begin{document}
\maketitle

\begin{abstract}
For a broad class of Fr\' echet--Lie supergroups $\Gsu$, we prove that there exists a correspondence between positive definite smooth (resp., analytic) superfunctions on $\Gsu$ and   
matrix coefficients of smooth (resp., analytic) unitary representations of the 
Harish--Chandra pair $(G,\g g)$ 
associated to $\Gsu$. 

As an application, we prove that a smooth positive definite superfunction on $\Gsu$ is analytic if and only if it restricts to an analytic function on the underlying manifold of $\Gsu$.

When the underlying manifold of $\Gsu$ is 1-connected  we obtain a necessary and sufficient condition for a linear functional on the universal enveloping algebra 
$U(\g g_\C)$ to correspond to a matrix coefficient of a unitary representation of $(G,\g g)$. 

The class of Lie supergroups for which the aforementioned results hold is characterised by a condition on the convergence of the Trotter product formula. This condition is strictly weaker than assuming that the underlying Lie group  of $\Gsu$ is a locally exponential Fr\'echet--Lie group. In particular, our results apply to examples of interest in representation theory such as  mapping supergroups and diffeomorphism supergroups.
\end{abstract}

\section{Introduction}

The study of unitarizable modules of infinite-dimensional Lie superalgebras has a long history. 
Both physicists and mathematicians have obtained several interesting examples and classification results 
for these Lie superalgebras. Examples include the $N=1$ and $N=2$ super Virasoro algebras
\cite{boucheretal}, 
\cite{friedanqiu},
\cite{goddard}, \cite{iohara, iohara8}, superconformal current algebras \cite{kactodor}, and affine Lie superalgebras
\cite{jakobsen2}, \cite{jarzh}.


In \cite{varadarajan} the authors initiate harmonic analysis on Lie supergroups by
laying a precise mathematical foundation to
study unitary representations of finite-dimensional Lie supergroups, and use it to classify irreducible unitary representations of translation (and in particular,  Poincar\'e) Lie supergroups. Their main idea is to use the equivalence between the category of Lie supergroups and the category of \emph{Harish--Chandra pairs}
 \cite{kostant}, \cite{koszul}. 
A Harish--Chandra pair is a pair $(G,\g g)$ where 
$G$ is a Lie group,
$\g g = \g g_\eev \oplus \g g_\ood$ is a Lie superalgebra, $\g g_\eev = \Lie(G)$, and there is an \emph{adjoint} action of $G$ on $\g g$ (see Definition
\ref{def-blsupergroup}).
 To justify the robustness of the category of representations of Harish--Chandra pairs, 
 one needs a nontrivial \emph{stability} result which, when $\g g$ is a finite-dimensional Lie superalgebra, 
is proved in \cite[Prop. 2]{varadarajan}.

Several technical issues arise in the extension of the stability result of \cite{varadarajan} to the infinite-dimensional case. These technical issues are resolved in \cite{menesa} when $\g g$ is a Banach--Lie superalgebra. Nevertheless, 
many infinite-dimensional Lie supergroups which are interesting from the point of view of representation theory, such as mapping supergroups and diffeomorphism supergroups, are not Banach-Lie groups.
In \cite{nsfreshetsuper} we succeeded in extending the stability theorem to Harish--Chandra pairs $(G,\g g)$ where $\g g$ is a Fr\' echet--Lie superalgebra and $G$ has the Trotter property, that is, 
for every $x,y\in\Lie(G)$,
\[
\exp(t(x+y))=\lim_{n\to\infty}
\big(
\exp\big(\frac{t}{n}x\big)
\exp\big(\frac{t}{n}y\big)
\big)^n
\]
holds
in the sense of uniform convergence on compact subsets of $\R$. The latter class of Harish--Chandra pairs is broad enough to include the examples of interest in representation theory (see Example \ref{ex:trot}).

Lie supergroups such as 
mapping and diffeomorphism supergroups are infinite dimensional supermanifolds 
modeled on Fr\' echet 
spaces. In
\cite[Rem. 2.6]{delignemorgan}, it is pointed out that
the
Berezin--Kostant--Leites theory, which defines a supermanifold as a locally ringed space,
is not suitable in the infinite-dimensional context. Thus, infinite dimensional Lie supergroups 
should be considered 
as group-objects in a different category. 
The definition and properties of this category were initially outlined in a preprint by Molotkov and later studied extensively in Sachse's thesis 
\cite{sachse}
(see  \cite{aldlau} as well).
The idea behind the definition of the latter category  is  
the functor of points approach adapted to the framework of the DeWitt topology.  
In this approach, a supermanifold is uniquely determined by its $\Lambda_n$--points, where $\Lambda_n$ denotes the Gra\ss mann algebra with $n$ generators. Therefore 
a supermanifold can be thought of as a functor 
$\Fsu:\Gr\to\Man$ together with an atlas
which is induced by a Grothendieck 
(pre-)topology on the category  $\Man^\Gr$.
Here $\Gr$ is the category of finite-dimensional 
Gra\ss mann algebras and $\Man$ is the category of smooth or analytic manifolds
modeled on locally convex spaces.

\subsection{Our main results}
In this article we  investigate the relationship
between smooth (and analytic) positive definite superfunctions on a (possibly in\-fi\-nite-di\-men\-sio\-nal) Lie supergroup $\Gsu$, 
unitary representations of the 
Harish--Chandra pair $(G,\g g)$ associated to $\Gsu$, and positive linear functionals on the universal enveloping algebra $U(\g g_\C)$. 

Our first main result (Theorem \ref{cinfandhom}) identifies 
the $\C$--superalgebra of smooth superfunctions on 
$\Gsu$ with a natural subalgebra of  
$\homrm_{\g g_\eev}(U(\g g_\C),C^\infty(G,\C))$.
This is a well known result for Berezin--Kostant--Leites Lie supergroups \cite{koszul} but its proof in the 
infinite dimensional setting 
requires new ideas because  
the supermanifold structure is not given 
by a sheaf of superalgebras anymore. Another issue in infinite dimensions is the lack of standard charts 
obtained by the exponential map. In order to prove
Theorem \ref{cinfandhom}, we need to use
several basic facts about the structure of infinite dimensional Lie supergroups and their left invariant 
 differential operators. We were unable to find a reference for these facts and therefore we have 
included detailed proofs. 
The reader who is familiar with Harish--Chandra pairs but not interested in the technical details of the functorial approach to Lie supergroups can  continue reading the paper from Section 
\ref{sec-gns}.

Our second main result (Theorem \ref{gns-smth})
is that a smooth (resp., analytic) positive definite superfunction on $\Gsu$ is the matrix coefficient of a cyclic smooth (resp., analytic) unitary representation of $(G,\g g)$ and vice versa. This result is an extension of the well known Gelfand--Naimark--Segal (GNS) construction to Lie 
supergroups. In a sense it means that to describe unitary representations of Lie supergroups it is sufficient to study unitary representations of their Harish--Chandra pairs.
From the GNS construction we also obtain the following interesting corollary  
(see Corollary \ref{cor-autoan}): if $f$ is a smooth positive definite superfunction
on $\Gsu$ which restricts to an analytic function on the underlying Lie group $\Gsu_{\Lambda_0}$, then $f$ restricts to  an analytic function on the Lie group $\Gsu_\Lambda$ for every $\Lambda$.

Our method to prove Theorem \ref{gns-smth} is similar in spirit to the classical GNS construction, but several technical issues arise. For instance, unlike the classical GNS construction,
in our framework one has to work with unbounded representations of semigroups, such as the  
``semidirect product'' $G\ltimes U(\fg_\C)$, so that we are actually dealing with 
structures similar to crossed product algebras. 
The stability result of \cite{nsfreshetsuper} plays a crucial role in our argument.

Our third main result (Theorem \ref{thm-cinteg}) is about an extension of the noncommutative moment problem to Lie supergroups. For an elaborate discussion of the history of the noncommutative moment problem see \cite[Sec. 1]{neebanalytic}.
If $(\pi,\rho^\pi,\mathscr H)$ is an analytic unitary representation of $(G,\g g)$ and $v\in\mathscr H^\omega$, then one can construct a $\C$--linear map $\lambda_v:U(\g g_\C)\to\C$ defined by 
$\lambda_v(D):=\langle \rho^\pi(D)v,v\rangle$. The noncommutative moment problem is to characterize the $\C$--linear maps $\lambda_v:U(\g g_\C)\to\C$ 
which are  obtained from unitary representations by the above construction. We obtain a necessary and sufficient condition when $\g g$ is a Fr\' echet--Lie superalgebra and $G$ is a 1-connected Fr\' echet BCH--Lie group. This is achieved by modifying the method of \cite{neebanalytic} using the ideas that were developed in \cite{menesa}.

In future work, which will rely on this article and \cite{nsfreshetsuper}, we will study global realizations of unitarizable super Virasoro algebras 
(\cite{iohara}), superconformal current algebras (\cite{kactodor}), and mapping superalgebras (\cite{jarzh}).

\subsection{Structure of this article}
Section~2 will review the background material concerning calculus on locally convex spaces. Section~3 will review the definition of the category of supermanifolds. In Section~4 we show how one can associate a Harish--Chandra pair to a supermanifold. In Section~5 we study left invariant differential operators on Lie supergroups and prove Theorem~\ref{cinfandhom}. Section~6  is devoted to the GNS construction. The necessary and sufficient condition for integrability of functionals is proved in Section~7.

\section{Calculus on locally convex spaces}

Unless stated otherwise, all vector spaces will be 
over $\R$. 
If $E$ and $F$ are  vector spaces
and 
$m\geq 1$ then 
$
\mathrm{Alt}_m(E,F)
$ will denote the vector space of
alternating $m$-linear maps 
$f:E^m
\to F.
$ For convenience we  set $\mathrm{Alt}_0(E,F):=F$.
If $E=E_\eev\oplus E_\ood$ is a $\Z_2$--graded vector space, then the parity of a homogeneous element $v\in E$ will be denoted by $|v|\in\{\eev,\ood\}$.
\subsection{Smooth maps between locally convex spaces}
Throughout this paper, by a \emph{locally convex space} we mean a Hausdorff locally convex topological vector space.

We quickly review basic definitions and properties of differentiable maps between locally convex spaces. For further details see \cite{hamilton}, \cite{milnor}, and \cite{neebjpn}.

\begin{definition}
Let $E$ and $F$ be 
locally convex spaces and $U\subseteq E$ be an open set. 
A map $h:U\to F$ is called \emph{differentiable} at $p\in U$ if 
the directional derivatives
\[ 
\dd h(p)(v):=
\dd_vh(p):=
\lim_{t\to 0}\frac{1}{t}(h(p+tv)-h(p))
\]
exist for all $v\in E$. The map $h$ is called $C^1$ if it is continuous, differentiable at every $p\in U$, and the map  
\[
\dd h:U\times E\to F\ ,\ (p,v)\mapsto
\dd_vh(p)
\]
is continuous.  If $k>1$ is an integer then a
continuous map $h:U\to F$ is called $C^k$ if 
the limit
\[
\dd^jh(p)(v_1,\ldots, v_j):=
\lim_{t\to 0}\frac{1}{t}
\left(
\dd^{j-1}h(p+tv_j)(v_1,\ldots,v_{j-1})
-\dd^{j-1}h(p)(v_1,\ldots, v_{j-1})
\right)
\]
exists
for every $1\leq j\leq k$ and every 
$(u,v_1,\ldots,v_j)\in U\times  E^j$, and the maps
\[
\dd^jh:U\times E^j
\to F\ ,\ (p,v_1,\ldots,v_j)\mapsto \dd^jh(p)(v_1,\ldots,v_j)
\] 
are continuous.
We call a map $h:U\to F$ \emph{smooth} when it is $C^k$ for every $k\geq 1$. 
\end{definition}

\begin{remark}
The above notion of smooth maps naturally leads to smooth manifolds, Lie groups, etc. For more details see \cite{glockner}.
\end{remark}

It is known that if $h:U\to F$ is $C^k$ then for every $p\in U$ the map 
\[
E^k\to F\ ,\ 
(v_1,\ldots,v_k)\mapsto \dd^kh(p)(v_1,\ldots,v_k)
\] 
is a continuous symmetric 
$k$-linear map. The Chain Rule holds in the following form: if $E,F,G$ are locally convex spaces, $U\subseteq E$ and $V\subseteq F$ are open, and $f:U\to F$ and $g:V\to G$ are $C^1$ maps such that 
$f(U)\subseteq V$, then $g\circ f:U\to G$ is also 
$C^1$ and 
\[
\dd_v(g\circ f)(p)=\dd g\big(f(p)\big)
\big(\dd_v f(p)\big)
\text{\ \,for every }p\in U\text{ and every }
v\in E. 
\]

The following lemma is sometimes called the 
Fa\`{a} di Bruno formula. Its proof is by induction on $n$. We omit the proof because the argument in 
the locally convex setting is the same as the one  
for finite-dimensional spaces.

\begin{lemma}
\label{faadibruno}
Let $E$ and $F$ be locally convex spaces, $U\subseteq E$ be open, and $V\subseteq\R^n$ be an open 0-neighborhood. Let $\mathscr P_n$ denote the collection 
of partitions of 
the set $\{1,\ldots,n\}$.
If  $f:U\to F$ and $g:V\to U$ are smooth maps
then
\begin{align*}
\frac{\partial}{\partial t_1}
\cdots&
\frac{\partial}{\partial t_n}
f\circ g(t_1,\ldots,t_n)\res{t_1=\cdots=t_n=0}=
\sum_{\{A_1,\ldots,A_k\}\in\mathscr P_n}
\!\!\!\!\!\!\dd^k f\big(g(0,\ldots,0)\big)
(
v_{A_1},\ldots,v_{A_k}
)
\end{align*}
where
\[
v_A:=
\frac{\partial}{\partial t_{a_1}}
\cdots
\frac{\partial}{\partial t_{a_\ell}}
g(t_1,\ldots,t_n)\res{t_1=\cdots=t_n=0}
\]
for every 
$A:=\{a_1,\ldots,a_\ell\}\subseteq \{1,\ldots, n\}$.
\end{lemma}


We briefly mention the definition of analytic maps between locally convex spaces. For a discussion of different notions of analyticity as well as pertinent references,  see
\cite[Sec. 2]{neebanalytic}.
\begin{definition}
Let $E$ and $F$ be locally convex spaces and $U\subseteq E$ be an open set. A continuous map $h:U\to F$ is called \emph{analytic} if for every $p\in U$ there exists an open 0-neighborhood $V_p$ in $E_\C:=E\otimes_\R\C$ 
and continuous homogeneous polynomials $h_n:E\to F$ of degree $n$ such that  
$h(p+v)=\sum_{n=0}^\infty h_n(v)$ for every $v\in V_p\cap E$.
\end{definition}

\subsection{Left invariant differential operators on Lie groups}
In this paper we assume that all Lie groups are smooth manifolds modeled on locally convex spaces.
The exponential map of a Lie group, if it exists, will be denoted by $x\mapsto e^x$.

Let $H$ be a Lie group and $\g h:=\Lie(H)$ be its Lie algebra. Assume that the exponential map 
\[
\g h\to H\ ,\ x\mapsto e^x
\]
is smooth. 
Let $U\subseteq H$ be open, $F$ be a locally convex space, and $h:U\to F$ be a smooth map. 
For every $x\in\g h$ set 
\[
\lddd_x h:U\to F\ ,\
\lddd_xh(g):=\lim_{t\to 0}\frac{1}{t}\big(
h(ge^{tx})-h(g)\big).
\]
\begin{lemma}
\label{lemsymmetr}
Let $H$ be a Lie group with a smooth exponential map,  $F$ be a locally convex space,  $h:H\to F$ be a smooth map, and 
$v_1,\ldots, v_n\in \Lie(H)$. For every $g\in H$
consider the map
\[
u_{g}:\R^n\to F\ ,\
u_{g}(t_1,\ldots,t_n):=
h(ge^{t_1v_1+\cdots+t_nv_n}).
\]
Then
\[
\frac{\partial}{\partial t_1}
\cdots
\frac{\partial}{\partial t_n}
u_{g}(t_1,\ldots,t_n)\res{t_1=\cdots=t_n=0}
=
\frac{1}{n!}
\sum_{\sigma\in S_n}\lddd_{v_{\sigma(1)}}\cdots\lddd_{v_{\sigma(n)}}h(g).
\]
\end{lemma}

\begin{proof}
Fix $ x_1,\ldots, x_n\in\R$ and set $y:=\sum_{i=1}^nx_iv_i$.
Consider the map
\[
\gamma:\R\to H\ ,\ \gamma(s):=ge^{sy}.
\] 
Then  
$
h\circ\gamma(s)=u_{g}(x_1s,\ldots,x_ns)
$
and therefore
for every 
$s_\circ\in\R$ we can write 
\begin{align*}
\sum_{i=1}^nx_i\lddd_{v_i}h(ge^{s_\circ y})
&=\lddd_yh(ge^{s_\circ y})
=\frac{\partial}{\partial s}(h\circ\gamma)
\res{s=s_\circ}\\
&
=\sum_{i=1}^nx_i
\Big(
\frac{\partial}{\partial t_i} 
u_{g}(
t_1,\ldots,t_n)
\res{t_1=x_1s_\circ,\ldots,t_n=x_ns_\circ}
\Big).
\end{align*}
To complete the proof
 we use the above relation repeatedly to compute 
$(\lddd_y)^nh(ge^{s_\circ y})$, 
 set $s_\circ=0$, and compare the coefficient of $x_1\cdots x_n$ on both sides.
\end{proof}

\section{Supermanifolds as functors}

A locally convex space $E$ is called \emph{$\Z_2$--graded} 
if $E=E_\eev\oplus E_\ood$ where $E_\eev$ and $E_\ood$ are locally convex spaces and the direct sum decomposition of $E$ is topological.

Throughout this section $E=E_\eev\oplus E_\ood$ and $F=F_\eev\oplus F_\ood$ will denote 
$\mathbb Z_2$--graded locally convex spaces.

\subsection{The category $\Gr$}
Let  $\mathsf{Gr}$ denote the category of finite-dimensional real Gra\ss mann algebras, i.e., unital associative
$\R$-algebras  
$\Lambda_n$, $n\geq 0$, generated by elements $\lambda_1,\ldots,\lambda_n$ which satisfy the relations $\lambda_i\lambda_j+\lambda_j\lambda_i=0$ for every $1\leq i, j\leq n$. Every $\Lambda\in\Gr$ has a canonical $\Z_2$--grading $\Lambda:=\Lambda_\eev\oplus \Lambda_\ood$, i.e., it is an associative superalgebra.
The identity element of any $\Lambda\in\Gr$ will be 
denoted by $1_{\Lambda}$.
Given an integer $n\geq 0$ and a set $I=\{i_1,\ldots,i_\ell\}\subseteq\N$ where $1\leq i_1<\cdots<i_\ell\leq n$, we define $\lambda_I\in\Lambda_n$ by $\lambda_I:=\lambda_{i_1}\cdots\lambda_{i_\ell}$.

The morphisms between objects of $\Gr$ are homomorphisms of $\Z_2$--graded unital algebras. For every $m,n\geq 0$ 
the set of morphisms from $\Lambda_m$ into $\Lambda_n$
will be denoted by $\hom_\Gr(\Lambda_m,\Lambda_n)$.

Observe that $\Lambda_0\simeq\R$ and therefore for every $\Lambda\in\Gr$ there exist unique morphisms $\varepsilon_\Lambda\in\hom_\Gr(\Lambda,\Lambda_0)$ and $\iota_\Lambda\in\hom_\Gr(\Lambda_0,\Lambda)$.
The kernel of $\varepsilon_\Lambda$ is called the \emph{augmentation ideal} of $\Lambda$ and will be 
denoted by $\Lambda^+$. We set $\Lambda^+_\eev:=\Lambda^+\cap\Lambda_\eev$.

For every $m\geq n\geq 0$ let $\varepsilon_{m,n}\in\hom_\Gr(\Lambda_m,\Lambda_n)$ and $\iota_{n,m}\in\hom_\Gr(\Lambda_n,\Lambda_m)$ be the homomorphisms uniquely identified by
\[
\varepsilon_{m,n}(\lambda_k):=
\begin{cases}
\lambda_k&\text{if }k\leq n,\\
0&\text{otherwise}
\end{cases}
\]
and 
$
\iota_{n,m}(\lambda_k):=\lambda_k 
$ 
for all $1\leq k\leq n$.
In particular
$\varepsilon_{m,0}=\varepsilon_{\Lambda_m}$
and $\iota_{0,m}=\iota_{\Lambda_m}$.

\subsection{The category $\Man^\Gr$}

Let $\Man$ denote the category of smooth manifolds modeled on locally convex spaces (see Remark \ref{remanalytictheory}). 
For every two functors $\Fsu,\Gsu:\mathsf{Gr}\to\Man$, the set of natural transformations from $\mathcal F$ to 
$\Gsu$ will be denoted by 
$\Natu(\Fsu,\Gsu)$.
The category whose objects are functors $\mathcal F:\mathsf{Gr}\to\Man$ and whose
morphisms are natural transformations will be denoted by
$\Funct$. If $\Fsu\in\Funct$ and $\varrho\in\hom_\Gr(\Lambda,\Lambda')$
then the morphism in 
$\Man$ from $\Fsu_\Lambda$ to $\Fsu_{\Lambda'}$ that is induced by $\varrho$ 
will be denoted by 
$
\Fsu_\varrho:\Fsu_\Lambda\to\Fsu_{\Lambda'}
$.

\begin{remark}
\label{ntrlinj}

Let $\Gsu\in\Man^\Gr$. For every $m\geq n\geq 0$ 
the map $
\Gsu_{\iota_{n,m}}:
\Gsu_{\Lambda_n}\to\Gsu_{\Lambda_m}
$ is injective and identifies 
$\Gsu_{\Lambda_n}$
with a subset of $\Gsu_{\Lambda_m}$.
We will use this natural 
identification to simplify our notation. 
For instance
if 
$
f\in\Natu(\Gsu,
\Fsu)
$
for some $\Fsu\in\Man^\Gr$, then for every 
$p\in\Gsu_{\Lambda_n}$
we write 
$f_{\Lambda_m}(p)$ 
instead of $f_{\Lambda_m}(\Gsu_{\iota_{n,m}}(p))$.
\end{remark}

Let $\Fsu,\Gsu\in\Funct$
and 
$\fphi\in\Natu(\mathcal G,\mathcal F)$. We write $\Gsu\sqsubseteq_\fphi\Fsu$ if	
for every $\Lambda\in\Gr$, the map 
$\fphi_\Lambda:\Gsu_\Lambda\to\Fsu_\Lambda$ 
is a diffeomorphism onto an open subset of 
$\Fsu_\Lambda$.
We write $\Gsu\sqsubseteq \Fsu$
if 
for every $\Lambda\in\Gr$ the set
$\Gsu_\Lambda$ is an open subset of $\Fsu_\Lambda$ and 
$\fphi_\Lambda:\Gsu_\Lambda\to\Fsu_\Lambda$ 
is the canonical injection.
If $\Hsu\sqsubseteq \Gsu$ and $f\in\Natu(\Gsu,\Fsu)$
then we define $f\res{\Hsu}\in\Natu(\Hsu,\Fsu)$ by 
\[
\left(f\res{\Hsu}\right)_\Lambda:=
f_\Lambda\res{\Hsu_\Lambda}.
\]

\subsection{Superdomains and supermanifolds}

Let $\EEsu\in\Funct$ 
be defined by 
\[
\EEsu_\Lambda:=(E\otimes \Lambda)_\eev
\text{ 
for every 
$\Lambda\in\Gr$.}
\]
The zero vector in $\EEsu_\Lambda$ 
will be denoted by $\zro_\Lambda$.
Every $\varrho\in\hom_\mathsf{Gr}(\Lambda,\Lambda')$ induces a map 
\[
\EEsu_\varrho:\EEsu_\Lambda\to\EEsu_{\Lambda'}
\ ,\ v\otimes \lambda\mapsto v\otimes\varrho(\lambda).
\]

\begin{definition}

A functor 
$\mathcal U\sqsubseteq\EEsu$
is called a \emph{superdomain}.

\end{definition}
The next proposition characterizes superdomains.
\begin{proposition}
\label{apx-opensubfcn}
If 
$\Usu\sqsubseteq\EEsu$ then there exists an open set $U\subseteq E_\eev$ such that 
\[
\Usu_\Lambda=U+(E\otimes \Lambda^+)_\eev
\] 
for every 
$\Lambda\in\Gr$.
\end{proposition}

\begin{proof}
Set $U:=\Usu_{\Lambda_0}$. 
Clearly
\[
\Usu_\Lambda
\subseteq 
\{
p\in\EEsu_\Lambda\,:\,\EEsu_{\eps_\Lambda}(p)\in U
\}
=
U+(E\otimes \Lambda^+)_\eev 
\]
for every $\Lambda\in\Gr$.
Next we prove that 
$
\Usu_\Lambda\supseteq U+(E\otimes \Lambda^+)_\eev
$.
Let $\Lambda:=\Lambda_n$. Fix $u\in U$. Then  
$\Usu_{\iota_\Lambda}(u)\in\Usu_\Lambda$, i.e.,  
$u\otimes 1_\Lambda\in \Usu_\Lambda$. Since 
$\Usu_\Lambda\subseteq \EEsu_\Lambda$ is open, there exists an open neighborhood $V$ of $\zro_\Lambda\in \EEsu_\Lambda$ such that 
$(u\otimes 1_\Lambda)+V\subseteq \Usu_\Lambda$. 

Set $V^+:=V\cap(E\otimes \Lambda^+)_\eev$.
For every $s>1$ let
$\varrho_s\in\hom_\mathsf{Gr}(\Lambda,\Lambda)$ 
be the homomorphism uniquely defined by 
$
\varrho_s(\lambda_i):=s\lambda_i\text{ for $1\leq i\leq n$}.
$ 
Then 
\[
\Usu_\Lambda
\supseteq
\Usu_{\varrho_s}
\left(
(u\otimes 1_\Lambda) +V^+
\right)
= (u\otimes 1_\Lambda) + 
\mathcal U_{\varrho_s}(V^+)
\]
and 
$
\bigcup_{s>1}
\left(
(u\otimes 1_\Lambda) + 
\Usu_{\varrho_s}(V^+)
\right)
=
 u\otimes 1_\Lambda+(E\otimes \Lambda^+)_\eev
$.
\qedhere
\end{proof}

\begin{definition}
Let $\Usu\sqsubseteq \EEsu$,
$\Vsu\sqsubseteq\EFsu$, and 
$\fphi\in\Natu(\Usu,\Vsu)$. We call $\fphi$ a
\emph{smooth morphism} from $\Usu$ to $\Vsu$  
if for every $\Lambda\in\Gr$ and every 
$p\in\Usu_\Lambda$ 
the differential map 
\[
\EEsu_\Lambda\to\EFsu_\Lambda
\ ,\
v\mapsto\dd_v\fphi_\Lambda(p) 
\] 
is $\Lambda_\eev$--linear.
\end{definition}

\begin{definition}
\label{def-fibrprod}
Let $\Usu,\Vsu,\Msu\in\Funct$, 
$\fphi\in\Natu(\Usu,\Msu)$, and 
$\fpsi\in\Natu(\Vsu,\Msu)$. Assume that
$\Vsu\sqsubseteq_\fpsi \Msu$. 
The \emph{fiber product} of 
$\Usu$ and $\Vsu$ over $\Msu$ is the functor 
$\Usu\times_\Msu\Vsu\sqsubseteq \mathcal U$
defined by 
\[
(\Usu \times_\Msu\Vsu)_\Lambda:=
\fphi_\Lambda^{-1}({\fpsi}_\Lambda(\mathcal V_\Lambda))
\text{ \,for every }\Lambda\in\Gr.
\]
We set
$
(\Usu\times_\Msu\Vsu)_\varrho:=
\Usu_\varrho\res{{(\Usu\times_\Msu\Vsu)}_\Lambda}
$ 
for every $\varrho\in\hom_\mathsf{Gr}(\Lambda,\Lambda')$.

\end{definition}
\begin{remark}
Let
$\Usu\sqsubseteq \EEsu$,
$\Vsu\sqsubseteq\EFsu$, 
$\fphi\in\Natu(\Usu,\Msu)$ and $\Vsu\sqsubseteq_g\Msu$. From the definition of fiber product it is easily seen that 
$\Usu\times_\Msu\Vsu\sqsubseteq\EEsu$. 
However, the canonical projection 
$p^\Vsu\in\Natu(\Usu\times_\Msu\Vsu,\Vsu)$ 
given by 
\[p^\Vsu_\Lambda:=g^{-1}_\Lambda\circ 
f_\Lambda\res{{(\Usu\times_\Msu\Vsu)}_\Lambda}
\text{ 
\, for every }\Lambda\in\Gr\]
might not necessarily be a smooth morphism.
\end{remark}
\begin{definition}
A  \emph{supermanifold} modeled on a $\Z_2$--graded locally convex space $E=E_\eev\oplus E_\ood$ is a pair $(\Msu,\mathscr A)$ where 
$\Msu\in\Funct$
and $\mathscr A$ is a set of pairs
$(\Usu,\fphi)$ 
satisfying
the following properties.
\begin{itemize}

\item[(i)] 
If $(\mathcal U,\fphi)\in \mathscr A$ then
$\Usu\sqsubseteq\EEsu$,
$
\fphi\in\Natu(\Usu,\Msu)
$,
and $\Usu\sqsubseteq_{\fphi}\Msu$.
\item[(ii)] $\mathscr A$ is an open covering of $\Msu$, i.e., 
\[
\Msu_\Lambda=\bigcup_{(\Usu,\fphi)\in\mathscr A} \fphi_\Lambda(\Usu_\Lambda)
\text{ \,for every }\Lambda\in\Gr.
\] 
\item[(iii)] For every two elements 
$(\Usu,\fphi)$ and $(\Vsu,\fpsi)$ of $\mathscr A$ the canonical projection
\[
p^\Vsu:\Usu\times_\Msu\Vsu\to\Vsu
\]
is a smooth morphism.
\end{itemize}
The set $\mathscr A$ is called an \emph{atlas} of $\Msu$. An 
element of $\mathscr A$ is called an 
\emph{open chart} of $\Msu$. 
\end{definition}
\begin{remark}
In the rest of this article, 
when there is no ambiguity about the atlas of a supermanifold $(\Msu,\mathscr A)$, 
we write $\Msu$ instead of 
$(\Msu,\mathscr A)$.
\end{remark}

\begin{lemma}
\label{-1m}
Let $\Msu$ be a supermanifold and 
$\mathcal U\sqsubseteq_f \mathcal M$. Then
\[
f_\Lambda(\mathcal U_\Lambda)=\Msu_{\eps_\Lambda}^{-1}(f_{\Lambda_0}(\mathcal  U_{\Lambda_0}))
:=
\{p\in\Msu_\Lambda\ :\ \Msu_{\eps_\Lambda}(p)\in f_{\Lambda_0}(\Usu_{\Lambda_0})
\}
\]
for every $\Lambda\in\Gr$.
\end{lemma}
\begin{proof}
Follows from Proposition \ref{apx-opensubfcn}. The proof is straightforward and left to the reader.
\end{proof}

The notion of a smooth morphism between two supermanifolds can now be defined using open charts. Let 
$(\Msu,\mathscr A)$ and $(\Nsu,\mathscr B)$ be two supermanifolds and $h\in\Natu(\Msu,\Nsu)$. 
For every 
$(\Usu,\fphi)\in\mathscr A$ and
every $(\Vsu,\fpsi)\in\mathscr B$
the natural transformations $h\circ\fphi\in
\Natu(\Usu,\Nsu)$ and 
$\fpsi\in\Natu(\Vsu,\Nsu)$ define a fiber product 
$\Usu\times_\Nsu\Vsu$.
We say $h$ is a
\emph{smooth morphism} from $\Msu$ to $\Nsu$ 
if for every 
$(\Usu,\fphi)\in\mathscr A$ and every
$(\Vsu,\fpsi)\in\mathscr B$ the canonical projection
$
p^\Vsu:\Usu\times_\Nsu\Vsu\to \Vsu
$
is a smooth morphism. 

\subsection{$\Lambda$--smooth maps}
The next definition will simplify our presentation.
\begin{definition}
\label{lsmoothfl}
Let $E=E_\eev\oplus E_\ood$ and $F=E_\eev\oplus F_\ood$ be $\Z_2$--graded locally convex spaces. Let 
$\Msu$ be a supermanifold modeled on $E$ and 
$\Nsu$ be a supermanifold modeled on $F$. Fix $\Lambda\in\Gr$ and let $p\in\Msu_\Lambda$.
A smooth map
$h_\Lambda:\Msu_\Lambda\to\Nsu_\Lambda$  is called \emph{$\Lambda$--smooth} at  $p$ if
for every open chart $(\Usu,\fphi)$ of $\Msu$ 
where  $p\in \fphi_\Lambda(\Usu_\Lambda)$,
and every open chart $(\Vsu,\fpsi)$ of 
$\Nsu$ where  
$h_\Lambda(p)\in \fpsi_\Lambda(\Vsu_\Lambda)$, the map
\[
\EEsu_\Lambda\to \EFsu_\Lambda
\ ,\ 
v\mapsto \dd_v\left(\fpsi_\Lambda^{-1}\circ
h_\Lambda\circ\fphi_\Lambda \right)(p)
\]
is $\Lambda_\eev$--linear.
If $h_\Lambda$ is $\Lambda$--smooth at every $p\in\Msu_\Lambda$ then $h_\Lambda$ will
simply be called $\Lambda$--smooth.
\end{definition}

\begin{remark}
Let $\Msu$ and $\Nsu$ be supermanifolds 
and $h\in\Natu(\Msu,\Nsu)$. Then $h$ is 
a smooth morphism if and only if $h_\Lambda:\Msu_\Lambda\to\Nsu_\Lambda$ is $\Lambda$--smooth for every $\Lambda\in\Gr$.
\end{remark}

\subsection{The category $\Sman$}
The category of supermanifolds and their smooth morphisms 
will be denoted by $\Sman$ and 
the set of smooth morphisms from 
$\Msu$ to $\Nsu$ will be denoted by 
$\hom_\Sman(\Msu,\Nsu)$.
It is proved in \cite[Thm 3.36]{aldlau} that the category of finite-dimensional supermanifolds (in the sense of Berezin, Kostant, and Leites) is equivalent to a full subcategory of $\Sman$.


\subsection{The superalgebra $C^\infty(\Msu,\FunC)$}
As usual, we 
denote the $\Z_2$--graded vector space $\C\oplus\C$ 
by $\C^{1|1}$. 
For every supermanifold $\mathcal M$ set
\[
C^\infty(\Msu,\FunC):=\hom_{\Sman}(\Msu,\FunC).
\]
The next proposition follows from \cite[Prop. 3.4]{aldlau}.

\begin{proposition}
\label{prpskeleton}
Let $\Usu\sqsubseteq \EEsu$ and set $U:=\Usu_{\Lambda_0}$. There exists a bijection between smooth morphisms $h\in C^\infty(\Usu,\FunC)$ and  families of maps 
$
\{h_m\ :\ m\geq 0\}
$
satisfying the following properties:
\begin{itemize}
\item[(i)] $h_m:U\to \mathrm{Alt}_m(E_\ood,\C)$ for every $m\geq 0$, and the map 
\[
U\times E^m_\ood\to F\ ,\ 
(u,v_1,\ldots,v_m)\mapsto h_m(u)(v_1,\ldots,v_m)
\]
is smooth.
\item[(ii)]
For every $\Lambda\in\Gr$ we have
\begin{equation}
\label{fdkfm}
h_\Lambda(u+v_\eev+v_\ood)=\sum_{k,m\geq 0}\frac{1}{k!m!}
\dd^kh_{m}(u)
(
\underbrace{v_\eev,\ldots,v_\eev}_{k\text{ times}}
)(\underbrace{v_\ood,\ldots,v_\ood}_{m\text{ times}}
)
\end{equation}
where $u\in U$,  $v_\eev\in E_\eev\otimes \Lambda^+_\eev$, and $v_\ood\in E_\ood\otimes \Lambda_\ood$.
\end{itemize}
In  \eqref{fdkfm}  the multilinear functions 
$\dd^kh_{m}(u)$ are extended by linearity to Grassmann variables, i.e.,
\begin{align*}
\dd^kh_m(u)(x_1\lambda_{I_1},\ldots,x_k\lambda_{I_k})&(
y_1\lambda_{J_1},\ldots,y_m\lambda_{J_m})\\
&:=
\dd^kh_m(u)(x_1,\ldots,x_k)(y_1,\ldots,y_m)\prod_{i=1}^k\lambda_{I_i}
\prod_{j=1}^m\lambda_{J_j}.
\end{align*}
\end{proposition}

\begin{remark}
The superspace $\FunC$ is a ring object in $\Funct$ and induces a 
multiplication 
on $C^\infty(\Msu,\FunC)$. Therefore $C^\infty(\Msu,\FunC)$ is the superalgebra of
smooth superfunctions on the supermanifold $\Msu$. 
In this article, we will not need this multiplication.\end{remark}
\begin{definition}
Let $\Msu$ be a supermanifold and 
$h\in C^\infty(\Msu,\FunC)$. The family of maps
$
\{h_m\ :\ m\geq 0\}
$ 
which satisfies \eqref{fdkfm} is called the \emph{skeleton} of $h$.
\end{definition}

\begin{lemma}
\label{hhnseries}
Let $\Usu\sqsubseteq \EEsu$,
$h\in C^\infty(\Usu,\FunC)$, and
$\{h_m\ :\ m\geq 0\}$  be the skeleton of $h$.  If $x_1,\ldots, x_n\in E_\ood$ then
\[
\dd^n h_{\Lambda_n}(u)(x_1\lambda_1,\ldots,x_n\lambda_n)
=
h_n(u)(x_1,\ldots,x_n)\cdot
\lambda_1\cdots\lambda_n
\,\text{ for every }u\in\Usu_{\Lambda_0}.
\]
\begin{proof}
If we set $v_\eev:=\zro_{\Lambda_n}$ and $v_\ood:=t_1x_1\lambda_1+\cdots+t_nx_n\lambda_n$ 
in 
\eqref{fdkfm} 
then it follows that 
$
h_{\Lambda_n}(u+t_1x_1\lambda_1+\cdots+t_nx_n\lambda_n)
$
is a vector-valued 
polynomial in $t_1,\ldots,t_n$ with coefficients in 
$\FunC_{\Lambda_n}$.
The leading term of this polynomial is 
\[
h_n(u)(x_1,\ldots,x_n)
\cdot\lambda_1\cdots\lambda_n
\cdot
t_1\cdots 
t_n.
\]
It follows that
\begin{align*}
\dd^n h_{\Lambda_n}(u)&
(x_1\lambda_1,\ldots,x_n\lambda_n)
\\
&=\frac{\partial}{\partial t_1}
\cdots
\frac{\partial}{\partial t_n}
h_{\Lambda_n}(u+t_1x_1\lambda_1+\cdots+t_nx_n\lambda_n)
\res{t_1=\cdots=t_n=0}\\
&=h_n(u)(x_1,\ldots,x_n)\cdot\lambda_1\cdots\lambda_n.
\qedhere
\end{align*}
\end{proof}
\end{lemma}

\begin{lemma}
\label{huefluehl}
Let $\Usu\sqsubseteq\EEsu$, 
$h\in\hom_\Sman(\Usu,\EFsu)$, 
$\Lambda\in\Gr$, and $p\in\Usu_\Lambda$.
\begin{itemize}
\item[(i)] If
$\varrho\in\hom_\Gr(\Lambda,\Lambda')$ for some $\Lambda'\in\Gr$ then 
\[
\dd^nh_{\Lambda'}
\big(
\Usu_\varrho(p)
\big)
\big(
\EEsu_\varrho(v_1),\ldots,
\EEsu_\varrho(v_n)
\big)=\EFsu_\varrho
\Big(
\dd^nh_\Lambda(p)(v_1,\ldots,v_n)
\Big)
\]
for every $v_1,\ldots,v_n\in\EEsu_\Lambda$.

\item[(ii)] The map
\[
\EEsu_\Lambda\times\cdots\times\EEsu_\Lambda
\to \EFsu_\Lambda
\ ,\ 
(v_1,\cdots,v_n)\mapsto 
\dd^n h_\Lambda(p)(v_1,\ldots,v_n)
\]
is $\Lambda_\eev$--linear in $v_1,\ldots,v_n$.
\end{itemize} 
\end{lemma}

\begin{proof}
(i) The case $n=1$ follows from differentiating the relation $h_{\Lambda'}\circ\Usu_\varrho=
\EFsu_\varrho\circ h_\Lambda$ and using the fact that
$\dd_v\Usu_\varrho(\zro_\Lambda)=\EEsu_\varrho(v)$ for every $v\in \EEsu_\Lambda$.
The case $n>1$ follows by induction on $n$.

(ii) The case $n=1$ is an immediate consequence of $\Lambda$--smoothness of $h_\Lambda$ 
and the case $n>1$ follows by induction on $n$.
\end{proof}

\begin{remark}
\label{remanalytictheory}
If $\Man$ denotes the category of analytic manifolds, then after minor modifications the definitions and results of this section will remain valid and therefore lead to the category of analytic supermanifolds. 
\end{remark}

\section{Lie Supergroups}
\label{sec-liesg}

Throughout this section $E=E_\eev\oplus E_\ood$ will be a Mackey complete $\Z_2$--graded locally convex space.
 For more information on the notion of Mackey completeness see
\cite[Sec. 2]{krieglmichor}.
Note that every sequentially complete locally convex space is Mackey complete.

\subsection{Lie supergroups as group objects in $\Sman$}
In view of our categorical approach to supergeometry,
a Lie supergroup will be a group object in the category $\Sman$. For background on group objects in categories 
see \cite[Sec. II.3.10]{gelfandmanin}. 

\begin{definition}
\label{defblgp}
A supermanifold 
$\Gsu\in\Sman$ is called a
\emph{Lie supergroup} if it
is a group object in $\Sman$ and
the Lie group $\Gsu_{\Lambda_0}$ has 
a smooth exponential map.

\end{definition}

In the rest of Section \ref{sec-liesg} we assume that
$\Gsu$ is a Lie supergroup modeled on $E=E_\eev\oplus E_\ood$.
If $\Lambda\in\Gr$ then  
$\Gsu_\Lambda$ is a Lie group modeled on the locally convex space $(E\otimes \Lambda)_\eev$. 
For every $\Lambda\in\Gr$, the identity element of $\Gsu_\Lambda$  will be denoted by $\yek_\Lambda$.

Let  $\mu:\Gsu\times\Gsu\to\Gsu$ denote the multiplication morphism of $\Gsu$. For every $\Lambda\in\Gr$, if $g\in \Gsu_\Lambda$ then we define
the left and right translation maps
$(l_g)_\Lambda:\Gsu_\Lambda\to\Gsu_\Lambda$
and 
$(r_g)_\Lambda:\Gsu_\Lambda\to\Gsu_\Lambda$
by
\[
(l_g)_\Lambda(x):=gx:=\mu_\Lambda(g,x)\text{ and }
(r_g)_\Lambda(x):=xg:=\mu_\Lambda(x,g).
\]
Note that $(l_g)_\Lambda$ and $(r_g)_\Lambda$
are $\Lambda$--smooth. If $g\in\Gsu_{\Lambda_0}$ then
based on Remark \ref{ntrlinj} we obtain smooth morphisms $l_g,r_g\in\hom_\Sman(\mathcal G,\mathcal G)$.

\subsection{A local Lie supergroup}

Fix an open chart $(\Vsu,\fphi)$ of $\Gsu$ such that $\Vsu\sqsubseteq \EEsu$. After a suitable 
translation we can assume that
$\zro_{\Lambda_0}\in\Vsu_{\Lambda_0}$, and 
$\fphi_{\Lambda_0}(\zro_{\Lambda_0})=\yek_{\Lambda_0}$. It follows immediately that $\zro_\Lambda\in\Vsu_\Lambda$ and $\fphi_\Lambda(\zro_\Lambda)=\yek_\Lambda$.

The multiplication and inversion of $\Gsu$ can be pulled back to 
$\Vsu$ as follows. Fix $\Lambda\in\Gr$. For every
$v_1,v_2\in\Vsu_\Lambda$ if
$
\mu_\Lambda(f_\Lambda(v_1),f_\Lambda(v_2))
\in 
f_\Lambda(\Vsu_\Lambda)
$
then we set
\begin{align}
\label{operbullet}
v_1\bullet v_2:=
f_\Lambda^{-1}\big(
\mu_\Lambda(f_\Lambda(v_1),f_\Lambda(v_2))
\big)
.
\end{align}
Similarly, for every $v\in\Vsu_\Lambda$, if $f_\Lambda(v)^{-1}\in f_\Lambda(\Vsu_\Lambda)$ then we set
\begin{align}
\label{operbullet2}
v^{-1}:=f_\Lambda^{-1}\big(f_\Lambda(v)^{-1}\big).
\end{align}
Note that the product $v_1\bullet v_2$  and 
the inverse $v^{-1}$ are not necessarily defined 
for all choices of $v_1,v_2,v\in \Vsu_\Lambda$. Nevertheless, the following statement holds.

\begin{lemma}
\label{VVVsubU}
There exists a $\Usu\sqsubseteq \mathcal V$ which satisfies the following properties:
\begin{itemize}
\item[(i)]
$\zro_{\Lambda}\in\Usu_{\Lambda}$ and $\zro_\Lambda\bullet u=u\bullet\zro_\Lambda=u$ 
for every $u\in\Usu_\Lambda$.
\item[(ii)] If $u\in\Usu_\Lambda$ for some $\Lambda\in\Gr$ then $u^{-1}$ is defined and
belongs to $\Usu_\Lambda$. Moreover, $u\bullet u^{-1}=u^{-1}\bullet u=\zro_\Lambda$.
\item[(iii)] The product $u_1\bullet u_2\bullet u_3$ is defined for every $\Lambda\in\Gr$ 
and for every $u_1,u_2,u_3\in\Usu_\Lambda$.
\end{itemize}
\end{lemma}
\begin{proof}
Since $\Gsu_{\Lambda_0}$ is a Lie group, it is possible to choose 
$\Usu_{\Lambda_0}\subseteq \Vsu_{\Lambda_0}$ suitably such that properties (i)-(iii) hold in the special case $\Lambda=\Lambda_0$. 
 For every $\Lambda\in\Gr$ set
\[
\Usu_\Lambda:=\Vsu_{\varepsilon_\Lambda}^{-1}
(\Usu_{\Lambda_0})
:=\{\,p\in\Vsu_\Lambda\ :\ \Vsu_{\varepsilon_\Lambda}(p)\in\Usu_{\Lambda_0}\,\}.
\]
Properties (i)-(iii) for general $\Lambda$ 
follow from naturality of $f$ and Lemma \ref{-1m}.
\end{proof}

\subsection{The exponential map for $\Gsu_\Lambda$}
Let $(\Usu,f)$ denote the open chart of $\Gsu$ obtained by Lemma \ref{VVVsubU}.
Lemma 
\ref{-1m} implies that 
the set 
\[
f_{\Lambda_{n+1}}^{-1}
\big(
\ker(\Gsu_{\eps_{n+1,n}})
\big)
=\big(E\otimes(\Lambda_n\cdot\lambda_{n+1})\big)_\eev
\subseteq \Usu_{\Lambda_{n+1}}
\]
is closed under the 
multiplication and inversion defined in \eqref{operbullet} 
and \eqref{operbullet2}. 
Thus, with the latter operations,
$\big(E\otimes(\Lambda_n\cdot\lambda_{n+1})\big)_\eev$
is
a Lie group. The next lemma explicitly 
describes 
the multiplication of $\big(E\otimes(\Lambda_n\cdot\lambda_{n+1})\big)_\eev$.
\begin{lemma}
\label{bulletisplus}
If $x,y\in \big(E\otimes(\Lambda_n\cdot\lambda_{n+1})\big)_\eev
$, then $x\bullet y=x+y$. 
\end{lemma}
\begin{proof}
Write $x=x_1\cdot\lambda_{n+1}$ and 
$y=y_1\cdot\lambda_{n+1}$ where 
$x_1,y_1\in (E\otimes\Lambda_n)_\ood$. 
Set $\varrho:=\varepsilon_{n+2,n+1}$, 
$x':=\Usu_{\iota_{n+1,n+2}}(x)
\in\Usu_{\Lambda_{n+2}}
$, and $y':=y_1\cdot\lambda_{n+2}\in\Usu_{\Lambda_{n+2}}$.
Then 
\[
\Usu_\varrho (x'\bullet y')= \Usu_\varrho(x')\bullet\Usu_\varrho(y')=
\Usu_\varrho(x')\bullet
\zro_{\Lambda_{n+1}}=\Usu_\varrho(x')
\]
and therefore 
$
x'\bullet y'
=
x'+w\cdot\lambda_{n+2}+
w'\cdot\lambda_{n+1}\lambda_{n+2}
$
where $w\in(E\otimes\Lambda_n)_\ood$
and $w'\in(E\otimes\Lambda_n)_\eev$.

Similarly if we take 
$
\varrho'\in\hom_\Gr(\Lambda_{n+2},\Lambda_{n+1})
$
defined by
\[
\varrho'(\lambda_i):=
\begin{cases}
\lambda_i&\text{ if }1\leq i\leq n,\\
0&\text{ if }i=n+1,\\
\lambda_{n+1}&\text{ if }i=n+2
\end{cases}
\]
then 
$
w\cdot\lambda_{n+1}=\Usu_{\varrho'}(x'\bullet y')=
\Usu(x')\bullet \Usu(y')=\zro_{\Lambda_{n+1}}\bullet
\Usu_{\varrho'}(y')
=y_1\cdot\lambda_{n+1}.
$
It follows that $w\cdot\lambda_{n+2}=y'$ and therefore $x'\bullet y'=x'+y'+w'\cdot\lambda_{n+1}\lambda_{n+2}$.
Finally, consider $\varrho''\in\hom(\Lambda_{n+2},\Lambda_{n+1})$ defined by
\[
\varrho''(\lambda_i):=\begin{cases}
\lambda_i&\text{ if }1\leq i\leq n,\\
\lambda_{n+1}&\text{ if }n+1\leq i\leq n+2.
\end{cases}
\]
Then 
\begin{align*}
x\bullet y=\Usu_{\varrho''}(x')\bullet\Usu_{\varrho''}(y')
=\Usu_{\varrho''}(x'\bullet y')=\Usu_{\varrho''}(x'+y'+w'\cdot\lambda_{n+1}\lambda_{n+2})
=x+y
\end{align*}
which completes the proof of the lemma.
\end{proof}

\begin{proposition}
\label{glhasexpsmth}
For every $n\geq 0$ 
the Lie group
$\Gsu_{\Lambda_n}$ has a smooth exponential map. 
\end{proposition}
\begin{proof}
We use induction on $n$. The case $n=0$ follows from Definition \ref{defblgp}.  Next assume $n>0$.
Since $E$ is Mackey complete, Lemma
\ref{bulletisplus} and \cite[Prop. II.5.6]{neebjpn} imply that $\ker(\Gsu_{\eps_{n+1,n}})$ is regular in 
the sense of
\cite[Def. 7.6]{milnor}. See  \cite[Def. II.5.2]{neebjpn} for a more detailed discussion.
If 
$
1\rightarrow A\rightarrow B\rightarrow C\rightarrow 1
$
is a smooth extension of Lie groups such that
$A$ is regular  and $C$ has a smooth exponential map, then $B$ also has a smooth exponential map. We do not give a proof of the latter statement because
the argument is very similar to 
\cite[Thm 38.6]{krieglmichor}. 
See also \cite{glocknerneeb}.
To complete the proof of the proposition we 
use the above statement with
the sequence
$
1\xrightarrow{\hspace{7mm}}
\ker(\Gsu_{\eps_{n+1,n}})
\xrightarrow{\hspace{7mm}}
\Gsu_{\Lambda_{n+1}}
\xrightarrow{\varepsilon_{n+1,n}}
\Gsu_{\Lambda_{n}}
\xrightarrow{\hspace{7mm}}
1
$.
\end{proof}

\subsection{The exponential map of $(\Usu,\fphi)$}
Let $(\Usu,\fphi)$ be the open chart of $\Gsu$ obtained by Lemma \ref{VVVsubU}.
Next we show how to pull back the exponential map of $\Gsu$ to $\Usu$. 
For every $\Lambda\in\Gr$, the map
\[
\EEsu_\Lambda\to\Lie(\mathcal G_\Lambda)
\ ,\ 
v\mapsto \dd_v f_\Lambda(\zro_\Lambda)
\]
is a continuous invertible linear transformation with a continuous inverse. 
For every 
$v\in \EEsu_\Lambda$ we set 
\[
e^v:=f_\Lambda^{-1}(e^{\dd_vf_\Lambda(\zro_\Lambda)})
\] 
whenever
$e^{\dd_vf_\Lambda(\zro_\Lambda)}\in 
f_\Lambda(\Usu_\Lambda)$.
The next lemma implies that the exponential map is well-defined on some $\Wsu\sqsubseteq \EEsu$.

\begin{lemma}
\label{expwecommu}
There exists a $\Wsu\sqsubseteq \EEsu$ which satisfies the following properties. \begin{itemize}
\item[(i)] $\zro_{\Lambda_0}\in\Wsu_{\Lambda_0}$.
\item[(ii)] $e^{\dd_vf_\Lambda(\zro_\Lambda)}\in 
f_\Lambda(\Usu_\Lambda)$
for every $\Lambda\in\Gr$ and every 
$v\in\Wsu_\Lambda$. 

\end{itemize}
\end{lemma}
\begin{proof}
Since 
$\Gsu_{\Lambda_0}$ has a smooth exponential map, there exists an open 0-neighborhood 
$W\subseteq\Lie(\mathcal G_{\Lambda_0})$ 
such that $\{e^w\ :\ w\in W\}\subseteq f_{\Lambda_0}(\Usu_{\Lambda_0})$. Set 
\[
\Wsu_{\Lambda_0}:=
\{\,v\in \EEsu_{\Lambda_0}\ :\ \dd_v f_{\Lambda_0}(\zro_{\Lambda_0})\in W\,\}
\]
and
$
\Wsu_\Lambda:=
\left(\EEsu_{\varepsilon_\Lambda}\right)^{-1}
(\Wsu_{\Lambda_0})
:=\{\,v\in\EEsu_\Lambda\ :\ 
\EEsu_{\varepsilon_\Lambda}(v)\in\Wsu_{\Lambda_0}\,
\}
$
for every $\Lambda\in\Gr$.

Next we prove that if $v\in\Wsu_\Lambda$ then
$e^{\dd_v f_\Lambda(\zro_\Lambda)}\in f_\Lambda(\Usu_\Lambda)$.
By Lemma \ref{-1m} it suffices 
to prove that if $v\in\Wsu_\Lambda$ then
$\Gsu_{\varepsilon_\Lambda}
\big(
e^{\dd_v f_\Lambda(\zro_\Lambda)}
\big)
\in f_{\Lambda_0}(\Usu_{\Lambda_0})$.
To prove the 
latter statement note that
if $v\in\Wsu_\Lambda$ then 
$\EEsu_{\varepsilon_\Lambda}(v)\in\Wsu_{\Lambda_0}$ and therefore
\begin{align*}
\Gsu_{\varepsilon_\Lambda}\big(
e^{\dd_v f_\Lambda(\zro_\Lambda)}
\big)&=
e^{
\dd\Gsu_{\varepsilon_\Lambda}(\yek_\Lambda)
(\dd_v f_\Lambda(\zro_\Lambda))
}=
e^{
\dd f_{\Lambda_0}(\zro_{\Lambda_0})
\left(\EEsu_{\varepsilon_\Lambda}(v)\right)
}\in
f_{\Lambda_0}(\Usu_{\Lambda_0}).
\qedhere
\end{align*}
\end{proof}
Lemma \ref{expwecommu} implies that for every $\Lambda\in\Gr$ the following diagram commutes:
\[
\xymatrix{
\Usu_\Lambda \ar[rr]^{f_\Lambda}  & &  \Gsu_\Lambda\\
\Wsu_\Lambda \ar[u]^{v\mapsto e^v}
\ar[rr]_{v\mapsto \dd_v f_\Lambda(\zro_\Lambda)\ \ }
&  & \Lie(\Gsu_\Lambda)
\ar[u]_{x\mapsto e^x}
}
\]
Our next goal is to prove that the 
induced exponential map is in
$\Natu(\mathcal W,\mathcal U)$. 
\begin{lemma}
\label{expandvarrho}
Let $\Lambda,\Lambda'\in \Gr$ and $\varrho\in\hom_\Gr(\Lambda,\Lambda')$. Then the following diagram commutes:
\[
\xymatrix{
\Usu_\Lambda \ar[rr]^{\Usu_\varrho}  & &  
\Usu_{\Lambda'}\\
\Wsu_\Lambda  \ar[u]^{v\mapsto e^v}
\ar[rr]_{\Wsu_\varrho}
&  & \Wsu_{\Lambda'}
\ar[u]_{v\mapsto e^v}
}
\]
\end{lemma}
\begin{proof}
Let $w\in\Wsu_\Lambda$. Then
\begin{align*}
f_{\Lambda'}\big(e^{\Wsu_\varrho(w)}\big)
&=
f_{\Lambda'}\big(e^{\EEsu_\varrho(w)}\big)
=
e^{\dd f_{\Lambda'}(\zro_{\Lambda'})(\EEsu_\varrho(w))}
=
e^{\dd \Gsu_\varrho(\yek_\Lambda)\left(
\dd_{w} f_\Lambda(\zro_\Lambda)
\right)}\\
&=
\Gsu_\varrho
\big(
e^{\dd_{w} f_\Lambda(\zro_\Lambda)}
\big)=\Gsu_\varrho
\big(
f_\Lambda(e^{w})
\big)
=f_{\Lambda'}(\Usu_\varrho(e^{w})).
\end{align*}
Since $f_{\Lambda'}:\Usu_{\Lambda'}\to\Gsu_{\Lambda'}$ is an injection, it follows that
$
e^{\Wsu_\varrho(w)}=
\Usu_\varrho(e^{w})
$.
\end{proof}

Fix $\Lambda\in\Gr$. The map 
\begin{equation}
\label{defsla}
S_\Lambda:\EEsu_\Lambda\to\Lie(\Gsu_\Lambda)
\ ,\ v\mapsto \dd_vf_\Lambda(\zro_\Lambda)
\end{equation}
is a continuous invertible linear transformation with a continuous inverse. Therefore
the Lie bracket of $\Lie(\Gsu_\Lambda)$ can be pulled back to $\EEsu_\Lambda$, that is, we can define
a Lie bracket 
\[
\EEsu_\Lambda\times\EEsu_\Lambda\to\EEsu_\Lambda\ ,\ 
[x,y]_\Lambda:=
S_\Lambda^{-1}
\left(
[S_\Lambda (x),
S_\Lambda (y)]_{\Lie(\Gsu_\Lambda)}
\right)
\text{ \ for every }x,y\in\EEsu_\Lambda,
\]
where $[\,\cdot,\cdot\,]_{\Lie(\Gsu_\Lambda)}$ denotes
the Lie bracket of $\Lie(\Gsu_\Lambda)$.
Observe that $e^{sx},e^{ty}\in\Wsu_\Lambda$ for sufficiently small $s,t\in\R$, and therefore
\begin{equation}
\label{xyldsdt}
[x,y]_\Lambda=\frac{\partial}{\partial s}
\frac{\partial}{\partial t}
\big(e^{sx}\bullet e^{ty}\bullet e^{-sx}\big)
\res{s=t=0}
\text{\ \ for every }x,y\in\EEsu_\Lambda.
\end{equation}
\begin{lemma}
\label{brllinrho}
Let $\Lambda\in\Gr$ and $x,y\in\EEsu_\Lambda$. 
\begin{itemize}
\item[(i)] If $\lambda_I,\lambda_J\in\Lambda_\eev$ then
$
[x\cdot\lambda_I,y\cdot\lambda_J]_\Lambda=
[x,y]_\Lambda\cdot\lambda_I\lambda_J.
$
\item[(ii)] If $\Lambda'\in\Gr$ and $\varrho\in\hom_\Gr(\Lambda,\Lambda')$ then 
\[
[\EEsu_\varrho(x),
\EEsu_\varrho(y)]_{\Lambda'}
=
\EEsu_\varrho\big(
[x,y]_\Lambda
\big)
.
\]
\end{itemize} 
\end{lemma}
\begin{proof}
Note that
\[
[x,y]_\Lambda=
\frac{\partial }{\partial s}
\frac{\partial }{\partial t}
h_\Lambda\circ g_\Lambda(s,t)\res{s=t=0}
\] where 
$h:\mathcal U\times 
\mathcal U\to \mathcal V
$ 
is the smooth morphism defined by 
\[
h_\Lambda(a,b):=a\bullet b\bullet a^{-1}\text{ for every }
a,b,c\in\Usu_\Lambda
\]
and 
\[
g_\Lambda:(-r,r)\times (-r,r)\to
\Usu_\Lambda\times\Usu_\Lambda\ ,\  
g_\Lambda(s,t):=(e^{sx},e^{ty})
\]
for $r>0$ sufficiently small.
Lemma \ref{faadibruno} implies that
\[
[x,y]_\Lambda=\dd^2 h_\Lambda
\big((\zro_\Lambda,\zro_\Lambda)\big)
\big((x,\zro_\Lambda),(\zro_\Lambda,y)\big).
\]
Therefore (i) and (ii) follow 
from Lemma \ref{huefluehl}.
\end{proof}

\subsection{A Lie superalgebra structure on $E$} Our next goal is to obtain a Lie superalgebra structure on $E=E_\eev\oplus E_\ood$ 
which is compatible with the Lie brackets $[\cdot,\cdot]_\Lambda$ on $\EEsu_\Lambda$ for every 
$\Lambda\in\Gr$.
\begin{lemma} 
\label{existenceofc1c2}
There exists a continuous bilinear map 
$
\bcye:E_\eev\times E_\ood\to E_\ood
$ 
such that 
\[
[x,y\cdot\lambda_1]_{\Lambda_1}=\bcye(x,y)\cdot\lambda_1
\text{ for every $x\in E_\eev$ and every 
$y\in E_\ood$.}
\]
There exists a continuous symmetric bilinear form 
$
\bcdo:E_\ood\times E_\ood\to E_\eev
$ 
such that
\[
[x\cdot\lambda_1,y\cdot\lambda_2]_{\Lambda_2}
=\bcdo(x,y)\cdot\lambda_1\lambda_2
\text{ for every }x,y\in E_\ood.
\]
\end{lemma}
\begin{proof}
This is a consequence of Lemma \ref{brllinrho}(ii). For $\bcye$ we use $\varrho:=\iota_{\Lambda_1}\circ\eps_{\Lambda_1}
\in\hom_\Gr(\Lambda_1,\Lambda_1)$. 
For $\bcdo$, 
we use $\varrho:=\varrho_i\in\hom_\Gr(\Lambda_2,\Lambda_2)$, where $i\in\{1,2\}$, defined by
$
\varrho_i(\lambda_j):=
\delta_{i,j}\lambda_j
$.
To show that $\bcdo$ is symmetric we use
$\varrho\in\hom_\Gr(\Lambda_2,\Lambda_2)$ given by 
$\varrho(\lambda_1):=\lambda_2$ and $\varrho(\lambda_2):=\lambda_1$.
\end{proof}
Consider the continuous skew-symmetric bilinear map
\[
\bcsfr: E_\eev\times E_\eev\to E_\eev\ ,\ \bcsfr(x,y):=[x,y]_{\Lambda_0}.
\]
Using $\bcsfr,\bcye,$ and $\bcdo$ we define a Lie superalgebra structure on $E=E_\eev\oplus E_\ood$. For every
$x_\eev,y_\eev\in E_\eev$ and every 
$x_\ood,y_\ood\in E_\ood$ we set 
\begin{equation}
\label{sbrc1c2}
[x_\eev+x_\ood,y_\eev+y_\ood]:=
\bcsfr(x_\eev,y_\eev)
+\bcye(x_\eev,y_\ood)
-\bcye(y_\eev,x_\ood)
-\bcdo(x_\ood,y_\ood).
\end{equation}
Bilinearity and continuity of $[\cdot,\cdot]$ are obvious. Example \ref{exsjac} below explains 
how to prove the super Jacobi identity for $[\cdot,\cdot]$ using the Jacobi identity of $[\cdot,\cdot]_\Lambda$. 

\begin{example}
\label{exsjac}
Assume $x,y\in E_\ood$ and $z\in E_\eev$. In this case we should prove that
\begin{equation}
\label{egxyz}
[x,[y,z]]-[y,[z,x]]+[z,[x,y]]=0.
\end{equation}
Using \eqref{sbrc1c2} we can write 
\eqref{egxyz} as 
\begin{equation}
\label{egxyz1}
-\bcdo(x,\bcye(z,y))-\bcdo(y,\bcye(z,x))
+\bcsfr(z,\bcdo(x,y))=0.
\end{equation}
From Lemma \ref{existenceofc1c2} and Lemma \ref{brllinrho}(ii) it follows that
\[
[y\cdot\lambda_1,z]_{\Lambda_2}=
-[z,y\cdot\lambda_1]_{\Lambda_2}
=-\bcye(z,y)\cdot\lambda_1.
\]
Again from Lemma \ref{existenceofc1c2} and Lemma \ref{brllinrho}(ii) it follows that 
\[
[x\cdot\lambda_2,[y\cdot\lambda_1,z]_{\Lambda_2}]_{\Lambda_2}
=
[x\cdot\lambda_2,-\bcye(z,y)\cdot\lambda_1]_{\Lambda_2}
=\bcdo(x,\bcye(z,y))\cdot\lambda_1\lambda_2.
\]
In a similar way, it can be shown that
$
[y\cdot\lambda_1,[z,x\cdot\lambda_2]_{\Lambda_2}]_{\Lambda_2}=
\bcdo(y,\bcye(z,x))\cdot\lambda_1\lambda_2
$
and
$
[z,[x\cdot\lambda_2,y\cdot\lambda_1]_{\Lambda_2}]_{\Lambda_2}=
-\bcsfr(z,\bcdo(x,y))\cdot\lambda_1\lambda_2
$.
Therefore \eqref{egxyz1} is a consequence of the Jacobi identity of $[\cdot,\cdot]_{\Lambda_2}$
for the elements $x\cdot \lambda_2,y\cdot\lambda_1,$ and $z$. 
\end{example}

For every $\Lambda\in\Gr$,
the Lie superbracket on $E=E_\eev\oplus E_\ood$ defined in \eqref{sbrc1c2} induces a Lie superbracket
\[
[\,\cdot,\cdot\,]'_\Lambda:
(E\otimes\Lambda)\times (E\otimes\Lambda) \to
E\otimes\Lambda
\]
 as follows. For homogeneous $v_1,v_2\in E$  and $\lambda_{I_1},\lambda_{I_2}\in\Lambda$ 
we set
\begin{equation*}
[v_1\cdot\lambda_{I_1},v_2\cdot\lambda_{I_2}]'_\Lambda
:=
(-1)^{|\lambda_{I_1}|\cdot|v_2|}[v_1,v_2]\cdot
\lambda_{I_1}\lambda_{I_2}.
\end{equation*}

\begin{proposition}
\label{prp-brkstsame}
Let $\Lambda\in\Gr$. Then
$
[x,y]_\Lambda=[x,y]'_\Lambda$ for every 
$
x,y\in\EEsu_\Lambda
$.
\end{proposition}
\begin{proof}
Follows from Lemma 
\ref{brllinrho} and 
Lemma \ref{existenceofc1c2}.
\end{proof}

\subsection{The Harish--Chandra pair associated to $\Gsu$}

For every $\Lambda\in\Gr$, we define an action of $\Gsu_{\Lambda_0}$ 
on $\EEsu_\Lambda$ as follows:
\begin{equation*}
\Gsu_{\Lambda_0}\times\EEsu_\Lambda\to\EEsu_\Lambda\ ,\
(g,v)\mapsto g\star_\Lambda v:=S_\Lambda^{-1}
\big(
\Ad_\Lambda(\Gsu_{\iota_\Lambda}(g))
(S_\Lambda (v))
\big).
\end{equation*}
where $S_\Lambda$ is the linear transformation defined in \eqref{defsla} and $\Ad_\Lambda(a)(x)$ denotes the adjoint action of $a\in\Gsu_\Lambda$ on $x\in\Lie(\Gsu_\Lambda)$. Note that
by the Chain Rule,
\begin{equation}
\label{addifrforms}
g\star_\Lambda v=
S_\Lambda^{-1}
\big(
\dd(l_g\circ r_{g^{-1}})_\Lambda(\yek_\Lambda)
(S_\Lambda(v))
\big)
=\dd(f^{-1}\circ l_g\circ r_{g^{-1}}\circ f)_\Lambda
(\zro_\Lambda)(v).
\end{equation}

\begin{lemma}
\label{lemprpstarl}
Let $g\in\Gsu_{\Lambda_0}$, $\Lambda\in \Gr$, and 
$v\in\EEsu_\Lambda$. 
\begin{itemize}
\item[(i)]
If $\varrho\in\hom_\Gr(\Lambda,\Lambda')$ for some $\Lambda'\in\Gr$ 
then 
$
g\star_{\Lambda'}\EEsu_\varrho(v)
=
\EEsu_\varrho(g\star_\Lambda v)
$.

\item[(ii)] 
$
g\star_\Lambda(v\cdot \lambda_I)
=
(g\star_\Lambda v)\cdot \lambda_I
$
for every 
$\lambda_I\in\Lambda_\eev$.

\end{itemize}

\end{lemma}

\begin{proof}
(i) Differentiating 
$
f_{\Lambda'}\circ\Usu_\varrho=
\Gsu_\varrho\circ f_\Lambda
$ 
yields
$
S_{\Lambda'}\circ \EEsu_\varrho(w)=
\dd\Gsu_\varrho(\yek_{\Lambda'})(S_\Lambda(w))
$. 
Similarly, differentiating
$
(l_g\circ r_{g^{-1}})_{\Lambda'}\circ \Gsu_\varrho
=\Gsu_\varrho\circ (l_g\circ  r_{g^{-1}})_\Lambda
$
yields
\[
\dd(l_g\circ r_{g^{-1}})_{\Lambda'}(\yek_{\Lambda'})
\big(
\dd\Gsu_\varrho(\yek_\Lambda)(w)
\big)
=
\dd\Gsu_\varrho(\yek_{\Lambda})
\big(
\dd(l_g\circ  r_{g^{-1}})_\Lambda
(\yek_\Lambda)(w)\big).
\]
Therefore
\begin{align*}
g\star_{\Lambda'}\EEsu_\varrho(v)
&=
S_{\Lambda'}^{-1}\Big(
\dd(l_g\circ r_{g^{-1}})_{\Lambda'}
(\yek_{\Lambda'})\big(
S_{\Lambda'}\circ\EEsu_\varrho(v)
\big)
\Big)\\
&=
S_{\Lambda'}^{-1}\Big(
\dd(l_g\circ r_{g^{-1}})_{\Lambda'}
(\yek_{\Lambda'})
\Big(
\dd\Gsu_\varrho(\yek_\Lambda)
(S_\Lambda(v))
\Big)
\Big)\\
&=
S_{\Lambda'}^{-1}
\Big(
\dd\Gsu_\varrho(\yek_{\Lambda})
\Big(
\dd(l_g\circ r_{g^{-1}})_\Lambda
(\yek_{\Lambda})\big(
S_\Lambda(v)
\big)
\Big)
\Big)\\
&=\EEsu_\varrho \circ S_\Lambda^{-1}
\Big(
\dd(l_g\circ r_{g^{-1}})_\Lambda
(\yek_{\Lambda})
\big(
S_\Lambda(v)
\big)
\Big)
=\EEsu_\varrho(g\star_\Lambda v).
\end{align*}

(ii) Follows
 from \eqref{addifrforms} and $\Lambda_\eev$--linearity
 of $v\mapsto \dd_v 
 (f^{-1}\circ l_g\circ r_{g^{-1}}
 \circ f)_\Lambda(\zro_\Lambda)$.
\end{proof}

\begin{theorem}
\label{liesg-hcpair} 
Let 
$G:=\Gsu_{\Lambda_0}$. Let 
$\g g_\eev:=\Lie(\Gsu_{\Lambda_0})$
and 
\[
A_\eev:E_\eev\to\g g_\eev\ ,\ 
A_\eev(e_\eev):=
\dd_{e_\eev}f_{\Lambda_0}(\zro_{\Lambda_0}).
\]
Let
$\g g_\ood:=
\ker\left(
\dd\Gsu_{\varepsilon_{\Lambda_1}}(\yek_{\Lambda_1})
\right)
\subseteq \Lie(\Gsu_{\Lambda_1})
$
and
\[
A_\ood:E_\ood\to
\g g_\ood
\ ,\ 
A(e_\ood):=
\dd f_{\Lambda_1}(\zro_{\Lambda_1})
\left({e_\ood\cdot\lambda_1}\right).
\]
Set $\g g:=\g g_\eev\oplus\g g_\ood$ and 
\[A:E\to\g g\ ,\ e_\eev\oplus e_\ood\mapsto
A_\eev (e_\eev)\oplus A_\ood (e_\ood).
\]
Let
$[\,\cdot,\cdot\,]_\g g:\g g\times \g g\to \g g$
be a Lie superbracket induced by 
the Lie superbracket $[\,\cdot,\cdot\,]:E\times E\to E$ given in \eqref{sbrc1c2} as follows:
\[
[x,y]_\g g:=A\big([A^{-1}(x),A^{-1}(y)]\big)
\text{ for every }x,y\in \g g.
\]
Define 
\[
\Ad:G\times \g g\to \g g
\ ,\ 
\Ad(g)(x_\eev\oplus x_\ood):=\Ad_{\Lambda_0}(g)(x_\eev)\oplus\Ad_{\Lambda_1}
\left(\Gsu_{\iota_{\Lambda_1}}(g)\right)
(x_\ood).
\]
\begin{itemize}
\item[(i)] 
The superbracket $[\,\cdot,\cdot\,]_\g g:\g g\times\g g\to\g g$
is continuous and 
the map $\Ad:G\times \g g\to \g g$ is smooth.
\item[(ii)] The map $\Ad(g):\g g\to \g g$ is an automorphism of $\g g $
for every $g\in G$.
\item[(iii)] Let $y\in\g g$ and $c_y:G\to \g g$ be the map defined by $c_y(g):=\Ad(g)(y)$. Then 
\[
\dd c_y(\yek_{\Lambda_0})(x)
=
[x,y]_\g g
\text{ for every $x\in \g g_\eev$.}
\]

\end{itemize}


\end{theorem}
\begin{proof}

(i) Continuity of $[\,\cdot,\cdot\,]_\g g$ follows from the fact that $A$ is a bijective continuous linear transformation with a continuous inverse. Smoothness of $\Ad$ follows from smoothness of $\Ad_{\Lambda_0}$ and $\Ad_{\Lambda_1}$.

(ii)
The inclusion
$\Ad(g)(\g g_\ood)\subseteq\g g_\ood$ 
follows from 
\[
\displaystyle
g\star_{\Lambda_0}\EEsu_{\eps_{\Lambda_1}}(v)
=
\EEsu_{\eps_{\Lambda_1}}\big(
g\star_{\Lambda_1}
v
\big)
\]
which is a consequence of Lemma \ref{lemprpstarl}(i).
Next we prove that 
\[
[\Ad(g)(x),\Ad(g)(y)]_\g g=\Ad(g)([x,y]_\g g)
\]
for every $x,y\in\g g$. It suffices to assume that 
$x,y$ are homogenous. Depending on the parity of $x$ and $y$ there are four cases to consider, but the arguments are similar, and we will only give the argument when $x,y\in\g g_\ood$. Set 
$\check x:=A^{-1}(x)$ and 
$\check y:=A^{-1}(y)$. 
From Lemma \ref{lemprpstarl} it follows that
\[
g\star_{\Lambda_2}(\check x\cdot\lambda_1)=
A^{-1}\big(\Ad(g)(x)\big)\cdot\lambda_1
\text{ \,and \,} 
g\star_{\Lambda_2}(\check y\cdot\lambda_2)=
A^{-1}\big(\Ad(g)(y)\big)\cdot\lambda_2.
\]
Now on the one hand by Proposition \ref{prp-brkstsame},
\begin{align*}
[A^{-1}(\Ad(g)(x))\cdot\lambda_1&,
A^{-1}(\Ad(g)(y))\cdot\lambda_2]_{\Lambda_2}\\
&=
-[A^{-1}(\Ad(g)(x)),A^{-1}(\Ad(g)(y))]\cdot\lambda_1\lambda_2
\end{align*}
and on the other hand from Lemma  \ref{lemprpstarl} it
follows that
\begin{align*}
[A^{-1}(\Ad(g)(x))\cdot\lambda_1&,
A^{-1}(\Ad(g)(y))\cdot\lambda_2]_{\Lambda_2}
=
[g\star_{\Lambda_2}(\check x\cdot\lambda_1),g\star_{\Lambda_2}(\check y\cdot\lambda_2)]_{\Lambda_2}\\
&=g\star_{\Lambda_2}\big(
[\check x\cdot\lambda_1,\check y\cdot\lambda_2]_{\Lambda_2}
\big)
=g\star_{\Lambda_2}(-[\check x,\check y]\cdot\lambda_1\lambda_2)\\
&=-(g\star_{\Lambda_2}[\check x,\check y])\cdot\lambda_1\lambda_2
=-A^{-1}\left(\Ad(g)
([x,y])\right)\cdot\lambda_1\lambda_2.
\end{align*}

(iii) It suffices to prove the statement for $y\in\g g_\ood$. In this case, 
\[
\dd c_y(\yek_{\Lambda_0})(x)
=
[\dd\Gsu_{\iota_{\Lambda_1}}(\yek_{\Lambda_0})(x),y]
_{\Lie(\Gsu_{\Lambda_1})}
=A\big([A^{-1}(x),A^{-1}(y)]\big)
=[x,y]_\g g
\]
which completes the proof.
\end{proof}
Theorem \ref{liesg-hcpair} links
the abstract notion of a Lie supergroup to the more concrete notion of a Harish--Chandra pair. 

In the next definition,
$\mathrm{Aut}(\g g)$ denotes the group of 
automorphisms of $\g g$ 
(not necessarily continuous)
which preserve the $\Z_2$--grading.
\begin{definition}
\label{def-blsupergroup}
A \emph{Harish--Chandra pair} is a pair $(G,\g g)$ satisfying the following properties.
\begin{itemize}
\item[(i)] $\g g:=\g g_\eev\oplus\g g_\ood$ is a
 $\Z_2$--graded locally convex space  endowed with a continuous Lie superbracket 
$[\cdot,\cdot]:\g g\times\g g\to \g g$.
\item[(ii)] $G$ is a Lie group and $\g g_\eev=\mathrm{Lie}(G)$. 
\item[(iii)] There exists a group homomorphism
$\Ad:G\to\mathrm{Aut}(\g g)$ 
such that the map 
\[
G\times \g g\to\g g \ ,\ (g,x)\mapsto \Ad(g)(x)
\] is smooth. 
\item[(iv)] 
If $c_y:G\to \g g$ is the map defined by $c_y(g):=\Ad(g)(y)$ then
\[
\dd_x c_y(\yek_G)=[x,y]\text{ for every }
x\in \g g_\eev\text{ and every }y\in\g g
\] 
where 
$\yek_G\in G$ denotes the identity element of $G$.
\end{itemize}
\end{definition}

\begin{corollary}
\label{cor-hcpr}
The pair $(G,\g g)$ associated to 
$\Gsu$ in
Theorem \ref{liesg-hcpair}
is a Harish--Chandra pair.

\end{corollary}

\begin{remark}
A result analogous to Corollary \ref{cor-hcpr} holds for analytic Lie supergroups, i.e., group objects in the category of analytic supermanifolds.
The Harish--Chandra pair $(G,\g g)$ associated to an analytic supermanifold satisfies two extra properties: $G$ is an analytic Lie group and the adjoint action $G\times \g g\to\g g$ is analytic.

\end{remark}
\begin{definition}
An \emph{analytic Harish--Chandra pair} is a Harish--Chandra pair $(G,\g g)$ 
where $G$ is an analytic Lie group and the adjoint action $G\times \g g\to \g g$ is analytic.
\end{definition}

\begin{remark}
From the results of this section it follows that 
$\Gsu_\Lambda\simeq G\ltimes N_\Lambda$ for every $\Lambda\in\Gr$, where $N_\Lambda$ is a nilpotent simply connected Lie group with Lie algebra $(\g g\otimes \Lambda^+)_\eev$. If we identify $N_\Lambda$ with its Lie algebra via the exponential map then the action of $G$ on $N_\Lambda$ is the canonical extension of the adjoint action of $G$  to $(\g g\otimes\Lambda^+)_\eev$.
\end{remark}

\section{Left invariant differential operators on Lie supergroups}

To simplify our notation, in this section we 
assume that $\Gsu$ is a Lie supergroup modeled on a locally convex space $\g g=\g g_\eev\oplus\g g_\ood$ (that is, 
$\g g_\eev=E_\eev$ and $\g g_\ood=E_\ood$) and  
$[\,\cdot,\cdot\,]:\g g\times\g g\to\g g$ is the Lie superbracket defined by \eqref{sbrc1c2}. 
Set $G:=\Gsu_{\Lambda_0}$. We will denote the identity element of $G$ by $\yek_G$.

Throughout this section
$(\Usu,f)$ will denote   
the open chart of $\Gsu$ obtained by Lemma 
\ref{VVVsubU} (note that 
$\Usu\sqsubseteq\Egsu$).

Let $h_\Lambda:\Gsu_\Lambda\to \FunC_\Lambda$ 
be a smooth map. Recall that for every 
$v\in\Lie(\Gsu_\Lambda)$ 
the left invariant differential operator $\lddd_v$ on $\Gsu_\Lambda$ is defined by
\[
\lddd_vh_\Lambda(g)
:=
\lim_{s\to 0}\frac{1}{s}
\big(
h_\Lambda(ge^{sv})-h_\Lambda(g)
\big).
\]
The chain rule implies that
$
\lddd_vh_\Lambda(g)
=
\dd h_\Lambda(g)
\big(
\dd_v(l_g)_\Lambda(\yek_\Lambda)
\big)
$.

\subsection{Some technical lemmas}
Our next goal is to prove some basic properties of left invariant differential operators.

\begin{lemma}
\label{gg0u0}
If $\Lambda\in\Gr$ then  every $g\in\Gsu_\Lambda$ can be written as 
$g=g_0f_{\Lambda}(u_0)$ where $g_0:=\mathcal G_{\iota_\Lambda\circ\varepsilon_\Lambda}(g)$ and 
$u_0\in\Usu_{\varepsilon_\Lambda}^{-1}(\zro_{\Lambda_0})\subseteq\Usu_\Lambda$.
\end{lemma}
\begin{proof}
Follows immediately from 
Lemma \ref{-1m}.
\end{proof}

\begin{lemma}
\label{proflv}
Let $\Lambda\in\Gr$ and $h_\Lambda:\Gsu_\Lambda\to \FunC_\Lambda$
be $\Lambda$--smooth.
\begin{itemize}
\item[(i)] For every $v\in\Lie(\Gsu_\Lambda)$ 
the map $\lddd_vh_\Lambda:\Gsu_\Lambda\to\FunC_\Lambda$ is $\Lambda$--smooth. 

\item[(ii)] 
Let $x_1,\ldots,x_k\in\Egsu_\Lambda$ and 
$\lambda_{I_1},\ldots,\lambda_{I_k}\in\Lambda_\eev$\,. 
Set $
\tilde x_i:=
\dd_{x_i\cdot\lambda_{I_i}} f_\Lambda(\zro_\Lambda)
$
for $1\leq i\leq k$. 
If
$g\in\Gsu_\Lambda$ then 
\[
\lddd_{\tilde x_1}
\cdots
\lddd_{\tilde x_k}
h_\Lambda(g)
=
\big(
\lddd_{\dd_{x_1} f_\Lambda(\zro_\Lambda)}
\cdots 
\lddd_{\dd_{x_k} f_\Lambda(\zro_\Lambda)}
h_\Lambda(g)\big)
\cdot\lambda_{I_1}\cdots\lambda_{I_k}.
\]

\end{itemize}
\end{lemma}

\begin{proof}
(i) Fix $g\in\Gsu_\Lambda$. By Lemma
\ref{gg0u0} we can write 
$g=g_0 f_\Lambda(u_0)$ 
where 
$
g_0:=\Gsu_{\iota_\Lambda\circ
\varepsilon_\Lambda}
(g)
$ and $u_0\in \Usu_\Lambda$. 
It suffices to prove that the map
\[
\Usu_\Lambda\to \FunC_\Lambda
\ ,\ 
x\mapsto (\lddd_vh_\Lambda)(g_0f_\Lambda(x))
\]
is $\Lambda$--smooth at $u_0$. Fix 
$w\in \Egsu_\Lambda$. Set $H_w(s,t):=h_\Lambda(g_0f_\Lambda(u_0+sw)e^{tv})$ for $s,t\in\R$ sufficiently close to 0.
Observe that for every $s\in\R$,
\[
\frac{\partial}{\partial t}H_w(s,t)\res{t=0}=
(\lddd_vh_\Lambda)(g_0f_\Lambda(u_0+sw)).
\]
Thus we need to prove that
$
\frac{\partial }{\partial s}
\frac{\partial}{\partial t}
H_w(s,t)\res{s=t=0}
$
is $\Lambda_\eev$--linear in $w$.
The map 
\[
\Usu_\Lambda\to \FunC_\Lambda
\ ,\ 
x
\mapsto
\big( h_\Lambda
\circ (l_{g_0})_\Lambda
\circ (r_{e^{tv}})_\Lambda
\circ f_\Lambda
\big)
(x)
\] 
is $\Lambda$--smooth 
and therefore
$\frac{\partial}{\partial s}H_w(s,t)\res{s=0}=\dd_w
\big(h_\Lambda
\circ (l_{g_0})_\Lambda
\circ (r_{e^{tv}})_\Lambda
\circ f_\Lambda\big)(u_0)$ is 
$\Lambda_\eev$--linear in $w$.
Consequently, 
if $\lambda\in\Lambda_\eev$ then
\begin{align*}
\frac{\partial }{\partial s}
\frac{\partial}{\partial t}
H_{w\cdot \lambda}(s,t)\res{s=t=0}
&=
\frac{\partial }{\partial t}
\big(
\frac{\partial}{\partial s}
H_{w\cdot\lambda}(s,t)
\res{s=0}
\big)
\res{t=0}\\
&=\frac{\partial}{\partial t}
\big(
\frac{\partial}{\partial s}H_w(s,t)
\res{s=0}
\big)
\res{t=0}\cdot\lambda
=
\frac{\partial }{\partial s}
\frac{\partial}{\partial t}
H_{w}(0,0)\cdot\lambda.
\end{align*}

(ii)  First assume $k=1$. 
By the Chain Rule we have
\begin{align*}
\lddd_{\dd_{x_1\cdot \lambda} f_\Lambda(\zro_\Lambda)}
h_\Lambda(g)
&=\dd h_\Lambda(g)
\Big(
\dd (l_g)_\Lambda(\yek_\Lambda)\big(\dd f_\Lambda(\zro_\Lambda)
(x_1\cdot \lambda)\big)
\Big)\\
&=
\dd_{x_1\cdot\lambda}
(h_\Lambda\circ (l_g)_\Lambda\circ f_\Lambda)
(\zro_\Lambda)=
\big(
\dd_{x_1}(
h_\Lambda\circ (l_g)_\Lambda\circ f_\Lambda)
(\zro_\Lambda)
\big)\cdot\lambda
\end{align*}
and again by the Chain Rule the right hand side is equal to 
$
\lddd_{\dd_{x_1} f_\Lambda(\zro_\Lambda)}
h_\Lambda(g)\cdot\lambda
$. 
This completes the proof for the case $k=1$.
The case $k>1$ follows from (i), (ii), and induction on $k$.
\end{proof}

\begin{lemma}
\label{lamlam'rho}
Let $\Lambda,\Lambda'\in\Gr$, 
$\varrho\in\hom_\Gr(\Lambda,\Lambda')$,
$h\in C^\infty(\Gsu,\FunC)$,
and  $w_1,\ldots,w_k\in\Egsu_\Lambda$. 
For every $1\leq i\leq k$ set
$
\tilde w_i:=\dd f_{\Lambda'}(\zro_{\Lambda'})
(\Egsu_\varrho(w_i))
$.
If $g\in\Gsu_\Lambda$ then 
\[
\lddd_{\tilde w_1}\cdots\lddd_{\tilde w_k}
h_{\Lambda'}\big(\Gsu_\varrho(g)\big)
=\FunC_\varrho
\big(
\lddd_{\dd_{w_1} f_\Lambda(\zro_\Lambda)}
\cdots
\lddd_{\dd_{w_k} f_\Lambda(\zro_\Lambda)}
h_\Lambda(g)
\big).
\]

\end{lemma}
\begin{proof}
We only give the argument for $k=1$, as the case 
$k>1$ follows by induction on $k$.
Observe that
$
\Gsu_\varrho\circ f_\Lambda
=f_{\Lambda'}\circ\Usu_\varrho
$ 
and  
$
\dd_w \Usu_\varrho(\zro_\Lambda)=
\Egsu_\varrho(w)
$. 
Thus
\[
\dd f_{\Lambda'}(\zro_{\Lambda'})
\big(
\Egsu_\varrho(w_1)
\big)
=
\dd f_{\Lambda'}(\zro_{\Lambda'})
\big(
\dd \Usu_\varrho(\zro_\Lambda)(w_1)
\big)
=\dd\Gsu_\varrho(\yek_\Lambda)
\big(
\dd f_\Lambda(\zro_\Lambda)(w_1)
\big)
\]
and  for every $s\in\R$ we have
\begin{align*}
e^{
s\dd f_{\Lambda'}(\zro_{\Lambda'})
\left(
\Egsu_\varrho(w_1)
\right)
}
=
e^{
s\dd\Gsu_\varrho(\yek_\Lambda)
\left(
\dd f_\Lambda(\zro_\Lambda)(w_1)
\right)
}=
\Gsu_\varrho
\big(
e^{s\dd f_\Lambda(\zro_\Lambda)(w_1)}
\big).
\end{align*}
It follows that
\begin{align*}
\lddd_{\tilde w_1}h_{\Lambda'}
\big(\Gsu_\varrho(g)\big)&=\lim_{s\to 0}
\frac{1}{s}
\bigg(
h_{\Lambda'}
\big(
\Gsu_\varrho(g)e^{
s\dd f_{\Lambda'}(\zro_{\Lambda'})
\left(\Egsu_\varrho(w_1)\right)
}
\big)
-h_{\Lambda'}\big(\Gsu_\varrho(g)\big)
\bigg)\\
&
=
\lim_{s\to 0}
\frac{1}{s}
\bigg(
h_{\Lambda'}
\big(
\Gsu_\varrho
(ge^{
s\dd f_{\Lambda}(\zro_{\Lambda})
\left(w_1\right)
}
)\big)
-h_{\Lambda'}\big(\Gsu_\varrho(g)\big)
\bigg)\\
&
=
\lim_{s\to 0}
\frac{1}{s}
\bigg(
\FunC_\varrho
\Big( 
h_{\Lambda}
\big(
ge^{
s\dd f_{\Lambda}(\zro_{\Lambda})
\left(w_1\right)
}
\big)
-
h_{\Lambda}(g)
\Big)
\bigg)\\
&=
\FunC_\varrho\bigg( 
\lim_{s\to 0}
\frac{1}{s}
\Big(
h_{\Lambda}
\big(
ge^{
s\dd f_{\Lambda}(\zro_{\Lambda})
\left(w_1\right)
}
\big)
-
h_{\Lambda}(g)
\Big)
\bigg)\\
&=
\FunC_\varrho
\big(
\lddd_{\dd_{w_1}f_{\Lambda}(\zro_\Lambda)}h_{\Lambda}(g)
\big).
\qedhere\end{align*}
\end{proof}

\subsection{Left invariant differential operators on $\Gsu$} 
Our next goal is to define left invariant differential operators 
\[
\ldd_x:C^\infty(\Gsu,\FunC)\to C^\infty(\Gsu,\FunC)
\]
for every $x\in\g g$ (see Remark \ref{rem-diffarelocal}). First we define $\ldd_x$ when $x$ is homogeneous and then we extend it to all of $\g g$ by linearity. 

\begin{lemma}
\label{wguniq}
Let $h\in C^\infty(\Gsu,\FunC)$, $x\in \g g_\ood$,  and $n\geq 0$ be an integer. 
For every $m>n$ set 
$
\tilde x_m:=
\dd_{x\cdot\lambda_m}\fphi_{\Lambda_m}(\zro_{\Lambda_m})
$.
If $g\in\Gsu_{\Lambda_n}$ then there exists a unique $w_g\in\FunC_{\Lambda_n}$ such that
\[
\lambda_m\cdot
\big(
\FunC_{\iota_{n,m}}(w_g)
\big)
=
\lim_{t\to 0}\frac{1}{t}\big(h_{\Lambda_{m}}
(g e^{t\tilde x_m})-h_{\Lambda_{m}}(g)\big)
\text{ for every }m>n.
\]
\end{lemma}
\begin{proof}
(i) 
For every $m>n$ set 
\[
w_{g,m}:=
\lim_{t\to 0}\frac{1}{t}\big(h_{\Lambda_{m}}
(g e^{t\tilde x_m})-h_{\Lambda_{m}}(g)\big)
.
\]
For every $m'>m>n$ let $\varrho_{m,m'}\in\hom_\Gr(\Lambda_m,\Lambda_{m'})$
by defined by
\[
\varrho_{m,m'}(\lambda_i):=
\begin{cases}
\lambda_i&\text{ if }1\leq i\leq m-1,\\
\lambda_{m'}&\text{ if }i=m.
\end{cases}
\]
If $m'>m>n$ then 
$
\dd_{\tilde x_m}
\Gsu_{\varrho_{m,m'}}(\yek_{\Lambda_m})=\tilde x_{m'}
$
and therefore 
\[
\Gsu_{\varrho_{m,m'}}(e^{t\tilde x_m})
=
e^{\dd_{\tilde x_m}\Gsu_{\varrho_{m,m'}}(\yek_{\Lambda_m})}
=
e^{t\tilde x_{m'}}
\text{ for every $t\in\R$}
.
\] 
It follows that
$
h_{\Lambda_{m'}}
(ge^{t\tilde x_{m'}})
=
\FunC_{\varrho_{m,m'}}\big(h_{\Lambda_m}(g e^{t\tilde x_{m}})\big)
$
for every $g\in \Gsu_\Lambda$ and thus $w_{g,m'}=\FunC_{
\varrho_{m,m'}}(w_{g,m})$. To complete the proof of the Lemma
it is enough  to show that 
$\FunC_{\varepsilon_{n+1,n}}(w_{g,n+1})=0$.
To prove the latter statement note that
\begin{align*}
\Gsu_{\varepsilon_{n+1,n}}(e^{t\tilde x_{n+1}})
&=e^{t\dd\Gsu_{\varepsilon_{n+1,n}}(\yek_{\Lambda_{n+1}})(\tilde x_{n+1})
}
=
e^{t\dd\Gsu_{\varepsilon_{n+1,n}}(\yek_{\Lambda_{n+1}})\left(
\dd_{x\cdot\lambda_{n+1}}
\fphi_{\Lambda_{n+1}}(\zro_{\Lambda_{n+1}})\right)
}\\
&=e^{\dd f_{\Lambda_n}(\zro_{\Lambda_n})
\left(\FunC_{\varepsilon_{n+1,n}}(x\cdot\lambda_{n+1})\right)}
=\yek_{\Lambda_n}
\end{align*}
and thus for every $t\in\R$ we have
\begin{align*}
\FunC_{\varepsilon_{n+1,n}}
\big(
h_{\Lambda_{n+1}}
(ge^{t\tilde x_{n+1}})
-h_{\Lambda_{n+1}}(g)
\big)
&=h_{\Lambda_n}(g\Gsu_{\varepsilon_{n+1,n}}
(e^{t\tilde x_{n+1}}))-h_{\Lambda_n}(g)
=0.
\qedhere
\end{align*}
\end{proof}

\begin{definition}
Let $h\in C^\infty(\Gsu,\FunC)$ and  $\Lambda\in\Gr$. 
If  $x\in \g g_\eev$\, then set
$
\tilde x:=\dd_x\fphi_{\Lambda}
(\zro_{\Lambda})
$ 
and define
\[
(\ldd_x h)_\Lambda(g):=
\lim_{t\to 0}\frac{1}{t}\left(
h_\Lambda(g e^{t\tilde x})-h_\Lambda(g)
\right)
\text{ for every }g\in \Gsu_\Lambda.
\]
For $x\in\g g_\ood$ we define 
\[
(\ldd_xh)_{\Lambda_n}(g):=w_g
\]
where $w_g\in\FunC_{\Lambda_n}$ is given by Lemma 
\ref{wguniq}.

\end{definition}

\begin{proposition}
\label{phcxge}
Let 
$h\in C^\infty(\mathcal G,\FunC)$.
\begin{itemize}
\item[(i)] If $x\in\g g$ then 
$\ldd_xh\in C^\infty(\Gsu,\FunC)$. 
\item[(ii)] If $x\in\g g$ and $g\in\Gsu_{\Lambda_0}$ then $\ldd_x(h\circ l_g)=(\ldd_xh)\circ l_g$ for every $g\in\Gsu_{\Lambda_0}$.

\end{itemize}
\end{proposition}

\begin{proof}

(i) We can assume $x$ is homogeneous. If $x\in\g g_\eev$ then the statement follows from Lemma
\ref{proflv}(i) and Lemma \ref{lamlam'rho}.
If $x\in \g g_\ood$ then the statement follows from
Lemma
\ref{proflv}(i), Lemma \ref{lamlam'rho}, and the definition
of $\ldd_x$. 


(ii) Straightforward from the definition.\qedhere
\end{proof}


\begin{lemma}
\label{Dxvsdx}
Let 
$x_1,\ldots,x_k\in \g g$ be homogeneous. Let $\Lambda\in\Gr$
and 
$\lambda_{I_1},\ldots,\lambda_{I_k}\in\Lambda$ satisfy 
$|\lambda_{I_i}|=|x_i|$. Let $\tilde x_i:=
\dd_{x_i\cdot\lambda_{I_i}} 
f_\Lambda(\zro_\Lambda)$ 
for $1\leq i\leq k$. Let $m\geq 0$ be an integer. 
Assume that
\[
I_i\cap \{r\in \N\ :\ r\leq m\}=\varnothing
\text{\ \,for $1\leq i\leq k$.}
\]
If $g\in\Gsu_{\Lambda_m}$
then 
\[
\lddd_{\tilde x_1}
\cdots
\lddd_{\tilde x_k}
h_{\Lambda}(g)
=
\lambda_{I_k}\cdots\lambda_{I_1}\cdot\big(
(\ldd_{x_1}\cdots \ldd_{x_k}h)_{\Lambda_m}(g)\big).
\]
\end{lemma}

\begin{proof}
By Lemma \ref{proflv}(ii) the proof is reduced to the case where each $I_i$ has at most one element. 

Let $i_1<\cdots <i_\ell$ be such that $I_i=\{m_i\}$ 
if and only if 
$i\in \{i_1,\ldots,i_\ell\}$. 
There are two cases to consider.\\
\textbf{Case 1.} The $m_{i_j}$'s 
are pairwise distinct numbers. Using Lemma \ref{lamlam'rho} with a homomorphism
$\varrho:\Lambda\to\Lambda$ which suitably permutes the generators of $\Lambda$, we can assume that $m_{i_1}<\cdots<m_{i_\ell}$. The case $k=1$ follows from the definition of the differential operator $\ldd_x$. Next we prove the case $k>1$ by induction on $k$. Assume that $I_1\neq \varnothing$. (The argument for the case $I_1=\varnothing$ is similar.) It follows that $I_i\cap \{r\in\N\ :\ r\leq m_{i_1}\}=\varnothing$ for 
$2\leq i\leq k$.
Set 
$
\tilde h:=
\ldd_{x_2}\cdots\ldd_{x_k}h
$. 
Then
\begin{align*}
\lddd_{\tilde x_1}
\cdots
\lddd_{\tilde x_k}
h_{\Lambda}(g)
&=
\lim_{s\to 0}
\frac{1}{s}
\Big(
\lddd_{\tilde x_2}
\cdots
\lddd_{\tilde x_k}
h_\Lambda(ge^{s \dd_{\tilde x_1} f_\Lambda(\zro_\Lambda)})
-
\lddd_{\tilde x_2}
\cdots
\lddd_{\tilde x_k}
h_\Lambda(g)
\Big)\\
&=
\lim_{s\to 0}
\frac{1}{s}
\Big(
\lambda_{m_{i_\ell}}
\cdots
\lambda_{m_{i_2}}
\cdot
\big(\tilde h_{\Lambda_{i_1}}
(ge^{s \dd_{\tilde x_1} f_\Lambda(\zro_\Lambda)})
-
\tilde h_{\Lambda_{i_1}}(g)
\big)\Big)\\
&=
\lambda_{m_{i_\ell}}
\cdots
\lambda_{m_{i_1}}
\cdot
(\ldd_{x_1}\tilde h)_{\Lambda_{m_{i_1}}}(g)
=
\lambda_{m_{i_\ell}}
\cdots
\lambda_{m_{i_1}}
\cdot
(\ldd_{x_1}\tilde h)_{\Lambda_m}(g)
\\
&=
\lambda_{m_{i_\ell}}
\cdots
\lambda_{m_{i_1}}\cdot
(\ldd_{x_1}\cdots \ldd_{x_k}h)_{\Lambda_m}(g).
\end{align*}
\textbf{Case 2.} There exist $1\leq a<b\leq \ell$ such that $m_{i_a}=m_{i_b}$. In this case we prove that
$
\big(\lddd_{\tilde x_1}
\cdots
\lddd_{\tilde x_k}
\big)
h_{\Lambda}(g)=0
$.
This follows from Case 1 and taking
$\varrho:\Lambda_{m+\ell}\to\Lambda$ defined by
\[
\varrho(\lambda_{j})=
\begin{cases}
\lambda_j&\text{ if }1\leq j\leq m,\\
\lambda_{m_{i_j}}&\text{ otherwise}
\end{cases}
\]
in Lemma \ref{lamlam'rho}.
\end{proof}

\begin{remark}
\label{rem-diffarelocal}
Since the definition of $\ldd_x$ is local, for every
$\mathcal Y\sqsubseteq\Gsu$ one can restrict $\ldd_x$ to a differential operator $\ldd_x:C^\infty(\mathcal Y,\FunC)
\to C^\infty(\mathcal Y,\FunC)$. Our statements regarding properties of $\ldd_x$  can be adapted suitably 
to hold for the restriction of $\ldd_x$ to $C^\infty(\mathcal Y,\FunC)$.

\end{remark}

\subsection{Extending $\ldd_x$ to the universal enveloping algebra}
Let $\g g_\C:=\g g\otimes_\R \C$ and 
$U(\g g_\C)$ denote the universal enveloping algebra of $\g g_\C$.
The next lemma shows that the definition of $\ldd_x$ can be extended to every $x\in U(\g g_\C )$.

\begin{lemma}
\label{cmutrel}
If $x,y\in \g g$ are homogeneous then
\[
\big(
\ldd_x\ldd_yh-(-1)^{|x|\cdot|y|}\ldd_y\ldd_xh
\big)_\Lambda(g)=(\ldd_{[x,y]}h)_\Lambda(g)
\text{ for every }g\in \Gsu_\Lambda.
\]
\end{lemma}
\begin{proof}
We only give the argument for the case $x,y\in\g g_\ood$.
The remaining cases are similar. 
Let $\Lambda:=\Lambda_n$ and set
\[
\tilde x:=\dd_{x\cdot\lambda_{n+1}} \fphi_{\Lambda_{n+2}}(\zro_{\Lambda_{n+2}})
\text{\, and \,}
\tilde y:=\dd_{y\cdot\lambda_{n+2}} \fphi_{\Lambda_{n+2}}(\zro_{\Lambda_{n+2}}).
\]
Thus
$\tilde x,\tilde y\in\Lie(\Gsu_{\Lambda_{n+2}})$ 
and
$
[\tilde x,\tilde y]=
\dd f_{\Lambda_{n+2}}(\zro_{\Lambda_{n+2}})
({-[x,y]\cdot\lambda_{n+1}\lambda_{n+2}})
$.
From Lemma \ref{Dxvsdx} it follows that
\[
\lddd_{\tilde x}\lddd_{\tilde y}h_{\Lambda_{n+2}}(g)=
\lambda_{n+2}\lambda_{n+1}
\cdot(\ldd_x\ldd_yh)_\Lambda(g)=
-\lambda_{n+1}\lambda_{n+2}
\cdot(\ldd_x\ldd_yh)_\Lambda(g).
\]
Similarly, 
$
\lddd_{\tilde y}\lddd_{\tilde x}h_{\Lambda_{n+2}}(g)=
\lambda_{n+1}\lambda_{n+2}
\cdot(\ldd_y\ldd_xh)_\Lambda(g)$.
Therefore 
\begin{align*}
\lambda_{n+1}\lambda_{n+2}
\cdot&
\big(
(\ldd_x\ldd_yh)_\Lambda(g)+(\ldd_y\ldd_xh)_\Lambda(g)
\big)\\
&=-\big(\lddd_{\tilde x}\lddd_{\tilde y}
-\lddd_{\tilde y}\lddd_{\tilde x}\big)
h_{\Lambda_{n+2}}(g)
=-\lddd_{[\tilde x,\tilde y]}
h_{\Lambda_{n+2}}(g)\\
&=-\ldd_{-[x,y]}h_\Lambda(g)\cdot
\lambda_{n+1}\lambda_{n+2}
=\ldd_{[x,y]}h_\Lambda(g)\cdot
\lambda_{n+1}\lambda_{n+2}\\
&=\lambda_{n+1}\lambda_{n+2}\cdot
\ldd_{[x,y]}h_\Lambda(g)
\end{align*}
which implies that
$
(\ldd_x\ldd_yh)_\Lambda(g)+(\ldd_y\ldd_xh)_\Lambda(g)
=\ldd_{[x,y]}h_\Lambda(g)
$.
\end{proof}

\subsection{Differentiating the exponential map}
The next lemma will be used in the proof of Theorem \ref{cinfandhom}.

\begin{lemma}
\label{uexpeta}
Let $\Lambda\in\Gr$, 
$v_1,\ldots,v_n\in\g g$ be homogeneous, and 
$\lambda_{I_1},\ldots,\lambda_{I_n}\in\Lambda$ satisfy the following properties.
\begin{itemize}
\item[(i)] $|\lambda_{I_i}|=|v_i|$ for all $1\leq i\leq n$. 
\item[(ii)] $I_i\cap I_j=\varnothing$ for every 
$1\leq i\neq j\leq n$.
\end{itemize}
Then there exists a smooth map
$\eta:\Usu_{\Lambda_0}\to\g g$ such that  
\begin{equation}
\label{eq-exxppes}
\frac{\partial}{\partial t_1}
\cdots
\frac{\partial}{\partial t_n}
(ue^{t_1v_1\lambda_{I_1}+\cdots+t_nv_n\lambda_{I_n}})
\res{t_1=\cdots=t_n=0}
=
\eta(u)\cdot\lambda_{I_1}\cdots\lambda_{I_n}
\end{equation}
for every $u\in\Usu_{\Lambda_0}$. On the left hand side of \eqref{eq-exxppes} 
we use the  identification of $u\in\Usu_{\Lambda_0}$ with 
$\Usu_{\iota_\Lambda}(u)\in \Usu_\Lambda$ (see 
Remark \ref{ntrlinj}).

\end{lemma}
\begin{proof}
First observe that differentiation with respect to the $t_i$'s commutes with $\Egsu_\varrho$ for every 
$\varrho\in\hom_\Gr(\Lambda,\Lambda)$. Therefore by
Lemma \ref{expandvarrho} it is enough to prove that,
if 
$a\in I_i$ 
for some $1\leq i\leq n$, 
and 
$\varrho\in\hom_\Gr(\Lambda,\Lambda)$ is defined by
\[
\varrho(\lambda_k):=
\begin{cases}
\lambda_k&\text{ if }k\neq a,\\
0&\text{ otherwise,}
\end{cases}
\]
then for all sufficiently small $t_j$, $j\neq i$, we have
\begin{equation}
\label{eglkddti}
\Egsu_{{\varrho}}
\bigg(
\frac{\partial}{\partial t_i}
(ue^{t_1v_1\lambda_{I_1}+\cdots+t_nv_n\lambda_{I_n}})
\res{t_i=0}
\bigg)=0.
\end{equation}
Without loss of generality we can assume $i=1$. 
By Lemma~\ref{expandvarrho}
\begin{align*}
\Usu_{\varrho}
\big(
ue&^{t_1v_1\lambda_{I_1}+\cdots+t_nv_n\lambda_{I_n}}
\big)
=ue^{\Egsu_\varrho
(t_1v_1\lambda_{I_1}+\cdots+t_nv_n\lambda_{I_n})}\\
&\hspace{1.1cm}=ue^{\Egsu_{\varrho}
(t_2v_2\lambda_{I_2}+\cdots+t_nv_n\lambda_{I_n})}=
\Usu_\varrho
\big(
ue^{t_2v_2\lambda_{I_2}+\cdots+t_nv_n\lambda_{I_n}}
\big).
\end{align*}
Equality \eqref{eglkddti} now follows immediately.
\end{proof}

\subsection{The algebra of smooth superfunctions revisited}
\label{thealgrevis}

Let $C^\infty(G,\C)$ 
(resp., $C^\omega(G,\C)$)
be the space of smooth (resp., analytic)
complex-valued 
functions on $G$. For every $x\in\g g_\eev$ set
$\tilde x:=\dd_x f_{\Lambda_0}(\zro_{\Lambda_0})$. 
The space $C^\infty(G,\C)$
is a $\g g_\eev$--module via the action
\begin{equation}
\label{xcdtpsi}
x\cdot \psi(g):=\lddd_{\tilde x}\psi(g)
\end{equation}
for every $x\in\g g_\eev$,  every 
$\psi\in C^\infty(G,\C)$, and every 
$g\in G$.

Let $
\homrm_{\g g_\eev}(U(\g g_\C),C^\infty(G,\C))
$
denote the complex
vector space of $\C$--linear maps 
\[
\mathbf h:U(\g g_\C)\to C^\infty(G,\C)
\]
which satisfy 
\begin{equation}
\label{EQTXYDXTY}
\big(\mathbf h(xy)\big)(g)=\lddd_{\tilde x}(\mathbf h(y))(g)
\end{equation}
for every 
$
x\in\g g_\eev
$, every
$y\in U(\g g_\C)$, 
and every $g\in G$.

If $\mathbf h\in \homrm_{\g g_\eev}(U(\g g_\C),C^\infty(G,\C))$, then for every $n\geq 1$ 
we set
\begin{equation}
\label{ggGsmuth}
\mathbf h^{[n]}:\g g^n\times G\to \C
\ ,\ \mathbf h^{[n]}(x_1,\ldots,x_n,g):= 
\big(\mathbf h(x_1\cdots x_n)\big)(g).
\end{equation}
Set
\[
\higuc:=\big\{
\mathbf h \in \homrm_{\g g_\eev}(U(\g g_\C),C^\infty(G,\C))
\ :\ \mathbf h^{[n]}\text{ is smooth for every }n\geq 1
\big\}.
\]
Similarly, we define
\[
\higuo:=\big\{
\mathbf h \in \homrm_{\g g_\eev}(U(\g g_\C),C^\omega(G,\C))
\ :\ \mathbf h^{[n]}\text{ is analytic for every }n\geq 1
\big\}.
\]
\begin{remark}
Let $Y\subseteq G$ be an open set. 
Then the space $C^\infty(Y,\C)$ (resp.,
$C^\omega(Y,\C)$) is also a $\g g_\eev$--module via the action given in \eqref{xcdtpsi}.
In the definition of $\higuc$ (resp., $\higuo$),
if we substitute 
$C^\infty(G,\C)$ (resp., $C^\omega(G,\C)$)
by  $C^\infty(Y,\C)$ (resp.,  
$C^\omega(Y,\C)$),  then 
we obtain 
the space
$C^\infty(Y,\g g)$ (resp.,  $C^\omega(Y,\g g)$).

\end{remark}
\begin{theorem}
\label{cinfandhom}
Let $\mathcal Y\sqsubseteq \Gsu$ and $Y:=\mathcal Y_{\Lambda_0}$. The map 
\[
\Phi:
C^\infty(\mathcal Y,\FunC)
\to
C^\infty(Y,\g g)
\ ,\ \Phi(h)(x):=(\ldd_xh)_{\Lambda_0}
\]
is a $\C$--linear isomorphism.
\end{theorem}

\begin{proof}
Throughout the proof we assume that $\mathcal Y=\Gsu$.
Slight changes render the proof applicable to arbitrary $\mathcal Y\sqsubseteq \Gsu$.

The proof will be given in several steps.

\textbf{Step 1.}
Let $h\in C^\infty(\Gsu,\FunC)$. From
the definition of $\ldd_x$ for $x\in\g g_\eev$, Proposition 
\ref{phcxge},
and Lemma \ref{cmutrel} 
 it follows that 
$\Phi(h)\in\homrm_{\g g_\eev}(U(\g g_\C),C^\infty(G,\C))$.
Lemma \ref{Dxvsdx} and 
smoothness of the exponential map of $\Gsu_\Lambda$ (see Proposition \ref{glhasexpsmth}) imply that
$\Phi(h)^{[n]}$ is smooth for every $n\geq 1$, i.e., 
$\Phi(h)\in C^\infty(G,\g g)$.  
It remains to prove that $\Phi$ is a bijection.

\textbf{Step 2.}
Fix $\Lambda\in\Gr$. By Lemma~\ref{gg0u0} 
every $g\in\Gsu_\Lambda$ can be written as $g=g_0f_\Lambda(u_0)$ where $g_0=\Gsu_{\iota_\Lambda\circ\varepsilon_\Lambda}(g)$.
We can identify $G$ 
with $\mathrm{im}(\Gsu_{\iota_\Lambda})$ 
(see Remark \ref{ntrlinj})
and consequently we will  
denote $\Gsu_{\eps_\Lambda}(g)$ by $g_0$ as well.
For every $h\in C^\infty(\Gsu,\FunC)$ we set $\hck:=h\circ l_{g_0}\circ f$.
Observe that $\hck\in C^\infty(\Usu,\FunC)$.  
Proving that $\Phi$ is an injection amounts to showing that 
if $\Phi(h)=0$ then $\hck=0$ for every  $g\in\Gsu_\Lambda$.

\textbf{Step 3.} Let $h\in C^\infty(\Gsu,\FunC)$ such that $\Phi(h)=0$. Fix $g\in \Gsu_\Lambda$.
For every $v_1,\ldots,v_n\in\Egsu_\Lambda$ 
and every $u\in \Usu_{\Lambda_0}$
set 
\begin{align*}
\frA(u;v_1,&\ldots, v_n)
:=\frac{\partial}{\partial t_1}
\cdots
\frac{\partial}{\partial t_n}
h_\Lambda
\big(
g_0f_\Lambda(u)e^{
t_1\dd_{v_1} f_\Lambda(\zro_\Lambda)
+
\cdots
+
t_n\dd_{v_n} f_\Lambda(\zro_\Lambda)
}\big)\res{t_1=\cdots=t_n=0}
\end{align*}
where $u$ is identified with $\Usu_{\iota_\Lambda}(u)\in\Usu_\Lambda$.
From the definition of 
$\frA(u;v_1,\ldots, v_n)$
it follows immediately that the map 
\[
\Usu_{\Lambda_0}\times
\Egsu_\Lambda\times\cdots\times\Egsu_\Lambda
\to\FunC_\Lambda\ ,\ 
(u,v_1,\ldots,v_n)\mapsto\frA(u;v_1,\ldots,v_n)
\] 
is smooth. Moreover,  
$\frA(u;v_1,\ldots,v_n)$
is linear in  $v_1,\ldots,v_n$.
If we set $\tilde v_i:=\dd_{v_i}f_\Lambda(\zro_\Lambda)$ for $1\leq i\leq n$ then
Lemma \ref{lemsymmetr} implies that
\begin{equation}
\label{al1.ln}
\frA(u;v_1,\ldots, v_n)=\frac{1}{n!}
\sum_{\sigma\in S_n}
\lddd_{\tilde v_{\sigma(1)}}
\cdots
\lddd_{\tilde v_{\sigma(n)}}
h_\Lambda(g_0f_{\Lambda}(u)).
\end{equation}
Since $\Phi(h)=0$, 
Lemma \ref{Dxvsdx} implies that $\frA(u;v_1,\ldots,v_n)=0$ for every $u\in\Usu_{\Lambda_0}$ and every $v_1,\ldots,v_n\in \Egsu_\Lambda$.



\textbf{Step 4.}
Given a set  $A=\{m_1,\ldots,m_k\}\subseteq \mathbb N$,
for every $u\in\Usu_{\Lambda_0}$ and every $v_{m_1},\ldots,v_{m_k}\in
\Egsu_\Lambda$ 
set
\[
\frB(u;v_A):=
\frac{\partial}{\partial t_{m_1}}
\cdots
\frac{\partial}{\partial t_{m_k}}
\big(
ue^{t_{m_1}
v_{m_1}
+
\cdots
+
t_{m_k}
v_{m_k}}
\big)
\res{t_{m_1}=\cdots=t_{m_k}=0}.
\]
Note that the smooth map
\[
\Usu_{\Lambda_0}\times\Egsu_\Lambda\times
\cdots\times\Egsu_\Lambda
\to\Egsu_\Lambda
\ ,\ 
(u,v_{A})\mapsto\frB(u;v_{A})
\]
is $k$-linear in $v_{m_1},\ldots,v_{m_k}$. Moreover, if $A=\{m_1\}$ then 
$\frB(u;v_A)=\dd_{v_{m_1}} (l_u)_\Lambda(\zro_\Lambda)$.
By Lemma~\ref{faadibruno} we can write
\begin{align*}
\notag
\frA(u;v_1,\ldots,v_n)
&=
\frac{\partial}{\partial t_1}
\cdots
\frac{\partial}{\partial t_n}\hck_\Lambda(ue^{t_1v_1+\cdots+t_nv_n})
\res{t_1=\cdots=t_n=0}\\
&=\notag
\sum_{\{A_1,\ldots,A_k\}\in\mathscr P_n}
\dd^k \hck_\Lambda(u)
\big(
\frB(u;v_{A_1}),\ldots,\frB(u;v_{A_k})
\big)
\end{align*}
and therefore
\begin{align}
\label{triangsys}
\frA(u;v_1,\ldots,v_n)
&=
\dd^n \hck_\Lambda(u)(w_1,\ldots,w_n)\\
&+\sum_{k<n}
\sum_{\{A_1,\ldots,A_k\}\in\mathscr P_n}
\dd^k \hck_\Lambda(u)
\big(
\frB(u;v_{A_1}),\ldots,\frB(u;v_{A_k})
\big)
\notag
\end{align}
where $w_i=\dd_{v_i} (l_u)_\Lambda(\zro_\Lambda)$ for every 
$1\leq i\leq n$. 
The map 
\[
\Egsu_\Lambda\to\Egsu_\Lambda\ ,\ 
v\mapsto \dd_v(l_u)_\Lambda(\zro_\Lambda)
\]
is a bijective continuous linear map with a continuous inverse $v\mapsto \dd_v(l_{u^{-1}})_\Lambda(u)$. Thus,
$v_i=\dd_{w_i}(l_{u^{-1}})_\Lambda(u)$.
If we start from any $w_1,\ldots, w_n\in \Egsu_\Lambda$ and then recursively apply 
\eqref{triangsys},
we obtain a linear system 
with finitely many equations with an invertible triangular coefficient matrix. 
Since $\frA(u,v_1,\ldots,v_n)=0$ for every $n$ and every 
$v_1,\ldots, v_n\in\Egsu_\Lambda$, the linear system is homogeneous.
Therefore the unique solution to the linear system is 
the trivial solution. It follows that
\[
\dd^n\hck_\Lambda(u)(w_1,\ldots,w_n)=0\text{ for every }
u\in \Usu_{\Lambda_0} 
\text{ and every }w_1,\ldots,w_n\in\Egsu_\Lambda.
\]
In particular, if $\Lambda=\Lambda_n$ and we set
$w_i=x_i\lambda_i$ where $x_i\in \g g_\ood$ 
for every $1\leq i\leq n$, then 
Lemma 
\ref{hhnseries}
and Proposition \ref{prpskeleton} imply that $\hck=0$. This completes the proof of injectivity of $\Phi$.

\textbf{Step 5.}
We now proceed towards the proof of surjectivity of $\Phi$. 
Lemma \ref{uexpeta}, the linear systems obtained by
\eqref{triangsys}, and  \eqref{al1.ln}
lead to the following statement:

\textsf{Statement A.} Let $n\geq 1$ and $\mathscr P_n$ denote the collection of partitions of $\{1,\ldots,n\}$. Then,
for every $S:=\{A_1,\ldots,A_k\}\in\mathscr P_n$, there exists a family
\[
\mathscr C_S:=\{
\,\big(\frD_{S,i}(u;x_{A_1}),\ldots,\frD_{S,i}(u;x_{A_k})\big)\ :\ i=1,\ldots,c(S)
\,
\}
\]
of $k$-tuples of smooth maps 
\[
\Usu_{\Lambda_0}
\times\g g^{|A_j|}_\ood\to \g g\ ,\ 
(u,x_{A_j})\mapsto\frD_{S,i}(u;x_{A_j})
\]
where $|A_j|$ denotes the cardinality of $A_j$, such that the following statements hold:
\begin{itemize}
\item[(i)] $\frD_{S,i}(u;x_{A_j})$ is linear in $\{x_s\ :\ s\in A_j\}$. 

\item[(ii)] If $h\in C^\infty(\Gsu,\FunC)$, then
for every $x_1,\ldots,x_n\in \g g_\ood$ and every $u\in\Usu_{\Lambda_0}$ we have:
\begin{align}
\label{dnhlx1l1xnln}
(-1)&^{\frac{n(n-1)}{2}}\dd^n \hck_\Lambda(u)(x_1\lambda_1,\ldots,x_n\lambda_n)\\
&=
\hspace{-5mm}\sum_{
\stackrel{S\in \mathscr P_n}{S=\{A_1,\ldots,A_k\}}
}
\hspace{-2.5mm}
\sum_{i=1}^{c(S)}
\big(
\ldd_{\frD_{S,i}(u;x_{A_1})}
\cdots
\ldd_{\frD_{S,i}(u;x_{A_k})}
h\big)_{\Lambda_0}(g_0f_\Lambda(u))
\cdot
\lambda_1\cdots\lambda_n.
\notag
\end{align}

\item[(iii)] Let $k=n$, i.e.,  $S=\{A_1,\ldots,A_n\}$ such that
$A_j=\{j\}$ for every $1\leq j\leq n$. 
Then $c(S)=1$ and 
$\frD_{S,1}(u;x_{\{j\}})=\bfc(u;x_j)$ 
for every $1\leq j\leq n$ where the smooth map
\[
\Usu_{\Lambda_0}\times \g g_\ood\to\g g_\ood
\ ,\
(u,x)\mapsto \bfc(u;x)
\]
is defined by the equality 
$
\bfc(u;x)\cdot\lambda_1
=\dd_{x\cdot\lambda_1}(l_{u^{-1}})_{\Lambda_1}(u)
$.
Moreover, for every $u\in\Usu_{\Lambda_0}$ 
the map 
\[
\g g_\ood\to\g g_\ood\ ,\ x\mapsto \bfc(u;x)
\]
is a bijective continuous linear transformation
with a continuous inverse.
\end{itemize}
The proofs of (i)-(iii) are fairly straightforward.  Part (i) follows from Lemma \ref{uexpeta} and the fact that the $\frD_{S,i}(u;x_A)$'s are obtained by superpositions of the $\frB(u,x_B)$'s.
Part (ii) follows from \eqref{triangsys}
and \eqref{al1.ln}.
Part (iii) follows from \eqref{triangsys} and 
the fact that 
$v\mapsto \dd_v(l_{u^{-1}})_{\Lambda_1}(u)$ is 
a bijective continuous linear transformation with a continuous inverse (see Step 4 above).

\textbf{Step 6.}
For every 
$g_0\in \Gsu_{\Lambda_0}$ 
let $\Vgzu\sqsubseteq \mathcal G$ be defined by
\[
\Vgzu_\Lambda:=(l_{g_0}\circ f)_\Lambda(\Usu_\Lambda)
\text{ for every }\Lambda\in\Gr.
\] 
Fix $\mathbf h\in C^\infty(G,\g g)$. Our goal is to prove that 
$\mathbf h=\Phi(h)$ for some $h\in C^\infty(\Gsu,\FunC)$.
Lemma \ref{Dxvsdx}, equality \eqref{EQTXYDXTY}, and a standard glueing argument show that
it is enough to prove the following statement:

\textsf{Statement B.}
For every 
$g_0\in\Gsu_{\Lambda_0}$ 
there exists a unique smooth morphism 
\[
\htil\in C^\infty(\Vgzu,\FunC)
\] 
which satisfies
$
(\ldd_{x_1}\cdots\ldd_{x_n}\htil)_{\Lambda_0}(g_1)=(\mathbf h(x_1\cdots x_n))(g_1)
$
for every $n\geq 0$, every $x_1,\ldots, x_n\in\g g_\ood$,
and every $g_1\in \Vgzu_{\Lambda_0}$.

The proof of the uniqueness part of Statement B 
is similar to Steps 1--4 above. Next we give the proof of the existence part of Statement B.

Fix $g_0\in \Gsu_{\Lambda_0}$. For every $u\in\Usu_{\Lambda_0}$, every 
$n\geq 1$ and every $x_1,\ldots,x_n\in \g g_\ood$\, define
$\kgz_n(u)(x_1,\ldots,x_n)$ as follows:
\begin{align}
\label{tdmxa1ak}
\kgz_n(u)&(x_1,\ldots,x_n)\\:=
&\sum_{
\stackrel{S\in \mathscr P_n}{S=\{A_1,\ldots,A_k\}}
}
\sum_{i=1}^{c(S)}
(-1)^{\frac{n(n-1)}{2}}\mathbf h\big({\frD_{S,i}(u;x_{A_1})}
\cdots
{\frD_{S,i}(u;x_{A_k})}
\big)(g_0f_{\Lambda_0}(u)).
\notag
\end{align}
Set $\kgz_0(u):=\mathbf h(1_{U(\g g_\C)})(g_0f_{\Lambda_0}(u))$ for every $u\in\Usu_{\Lambda_0}$,
where $1_{U(\g g_\C)}$ denotes the identity element
of $U(\g g_\C)$. By Proposition
\ref{prpskeleton}
the family $\{\kgz_n\ :\ n\geq 0\}$ corresponds to 
a smooth morphism 
$
\kgz:\Usu\to \FunC.
$
Let $
\htil:\Vgzu\to\FunC$ be the unique smooth morphism which satisfies
\[
\kgz=\htil\circ l_{g_0}\circ f.
\]
Statement A and Lemma  \ref{hhnseries}
imply that
\begin{align}
\label{lcux1xn}
(-1)^{\frac{n(n-1)}{2}}&
\kgz_n(x_1,\ldots,x_n)
=
\big(
\ldd_{\bfc(u;x_1)}
\cdots
\ldd_{\bfc(u;x_n)}
h^{g_0}\big)_{\Lambda_0}(g_0f_{\Lambda_0}(u))\\
&+\sum_{
\substack{
S\in \mathscr P_n\\
S=\{A_1,\ldots,A_k\}\\k<n
}
}\ 
\sum_{i=1}^{c(S)}
\big(
\ldd_{\frD_{S,i}(u;x_{A_1})}
\cdots
\ldd_{\frD_{S,i}(u;x_{A_k})}
h^{g_0}\big)_{\Lambda_0}(g_0f_{\Lambda_0}(u))
\notag
\end{align}
and \eqref{tdmxa1ak}
implies that
\begin{align}
\label{txmda1ak2}
(-1)^{\frac{n(n-1)}{2}}&
\kgz_n(x_1,\ldots,x_n)
=\mathbf h\big(
\bfc(u;x_1)\cdots\bfc(u;x_n)
\big)(g_0f_{\Lambda_0}(u))\\
\notag
&+\sum_{
\substack{S\in \mathscr P_n\\
S=\{A_1,\ldots,A_k\}\\
k<n}
}
\sum_{i=1}^{c(S)}
\mathbf h\big({\frD_{S,i}(u;x_{A_1})}
\cdots
{\frD_{S,i}(u;x_{A_k})}
\big)(g_0f_{\Lambda_0}(u)).
\end{align}
Since the map 
$x\mapsto \bfc(u;x)$ is an invertible linear map 
(see Statement A), from \eqref{lcux1xn}, \eqref{txmda1ak2}, and 
Lemma
\ref{hhnseries}, 
and 
by induction on $n$ we can prove that 
\[
(\ldd_{x_1}\cdots\ldd_{x_n}\htil)_{\Lambda_0}
(g_0f_{\Lambda_0}(u))=(\mathbf h (x_1\cdots x_n ))(g_0f_{\Lambda_0}(u))
\]
for every $u\in\Usu_{\Lambda_0}$ and every $x_1,\ldots, x_n\in\g g_\ood$.
The proof of Statement B is now complete.
\end{proof}

\begin{remark}
Note that $C^\infty(Y,\g g)$ is an associative $\C$-superalgebra with the multiplication
\[
\big(\mathbf h\cdot \mathbf h'\big)(x)(g):=
\big(
\mathsf m\circ (\mathbf h\otimes \mathbf h')\circ\mathsf c(x)
\big)(g)
\]
where
$\mathsf m:C^\infty(Y,\C)\otimes C^\infty(Y,\C)\to C^\infty(Y,\C)$ denotes the standard pointwise multiplication and 
$\mathsf c:U(\g g_\C)\to U(\g g_\C)\otimes U(\g g_\C)$ 
denotes the standard  co-multiplication.
It can be shown that the map 
$\Phi$ of Theorem \ref{cinfandhom}
is an isomorphism of $\C$-superalgebras. We will not need this fact and therefore we omit its proof.
\end{remark}

\section{The GNS construction}
\label{sec-gns}
Theorem \ref{cinfandhom} allows us to substitute 
a Lie supergroup $\Gsu$ by its associated 
Harish--Chandra pair. Therefore in the rest of this article we will not need the functorial formalism of the previous sections and we can concentrate on Harish--Chandra pairs.
Throughout this section $(G,\g g)$ will 
denote a Harish--Chandra pair.

\subsection{Smooth and analytic unitary representations}

We recall the definition of smooth and analytic unitary representations of a Harish--Chandra pair (see \cite{varadarajan}, \cite{menesa}, and
\cite{nsfreshetsuper}).

\begin{definition}
\label{defi-smoothanalytic}
Let $(G,\g g)$  be a Harish--Chandra pair (resp., an analytic Harish--Chandra pair). A \emph{smooth unitary representation} 
(resp., an \emph{analytic unitary representation}) of $(G,\g g)$ is a triple $(\pi,\rho^\pi,\mathscr H)$ satisfying the following properties.
\begin{enumerate}
\item[(i)] $(\pi,\mathscr H)$ is a smooth (resp., analytic) unitary representation of $G$
on the $\mathbb Z_2$--graded Hilbert space $\mathscr H=\mathscr H_\eev\oplus\mathscr H_\ood$ such that for every $g\in G$, the operator
$\pi(g)$ preserves the $\mathbb Z_2$--grading.
\item[(ii)] $\rho^\pi:\g g\to\End_\C(\mathscr B)$ is a representation of the Lie superalgebra $\g g$, where
$\mathscr B=\mathscr H^\infty$ (resp., $\mathscr B=\mathscr H^\omega$).
\item[(iii)] For every $x\in\g g_\eev$, if $\dd\pi(x)$ denotes the infinitesimal generator of the one-parameter group $t\mapsto \pi(e^{tx})$, then
$\rho^\pi(x)=\dd\pi(x)\big|_\mathscr B$.
\item[(iv)] $e^{-\frac{\pi i}{4}}\rho^\pi(x)$ is a symmetric operator for every $x\in\g g_\ood$.
\item[(v)] Every element of the component group $G/G^\circ$ has a coset representative $g\in G$ such that $\pi(g)\rho^\pi(x)\pi(g)^{-1}=\rho^\pi(\Ad(g)x)$ for every $x\in\g g_\ood$.

\end{enumerate}

\end{definition}

\begin{lemma}
\label{lme-contofgnh}
Let $(G,\g g)$ be a Harish--Chandra pair, $(\pi,\rho^\pi,\mathscr H)$ be a smooth unitary representation of $(G,\g g)$, and $v\in\mathscr H^\infty$. Assume that $\g g$ is a 
Fr\' echet--Lie superalgebra. 
Then for every $n\geq 1$ the map 
\begin{equation*}
\g g^n\to \mathscr H\ ,\ 
(x_1,\ldots,x_n)\mapsto \rho^\pi(x_1)\cdots\rho^\pi(x_n)v
\end{equation*}
is continuous.

\end{lemma}

\begin{proof}
The proof is by induction on $n$. For $n=0$ there is nothing to prove. Let $n\geq 1$. First we prove that the map 
\begin{equation}
\label{mapg0ggg}
\g g_\eev\times\g g^{n-1}\to \mathscr H\ ,\ (x,x_1\ldots,x_{n-1})\mapsto
\rho^\pi(x)\rho^\pi(x_1)\cdots\rho^\pi(x_{n-1})v
\end{equation}
is continuous. 
For every nonzero $t\in\R $ define a map 
$
f_t:\g g_\eev\times \g g^{n-1}\to \mathscr H
$
by
\[ 
f_t(x,x_1,\ldots,x_{n-1}):=\frac{1}{t}
\left(
\pi(e^{tx})\rho^\pi(x_1)\cdots\rho^\pi(x_{n-1})
-\rho^\pi(x_1)\cdots\rho^\pi(x_{n-1})\right)
v.
\]
The maps $f_t$ are continuous from a Baire space into a metric space. Moreover,  
\[
\lim_{t\to 0}f_{t}(x,x_1,\ldots,x_{n-1})= \rho^\pi(x)\rho^\pi(x_1)\cdots\rho^\pi(x_{n-1})v.
\]
It follows from \cite[Ch. IX, \S 5, Ex. 22(a)]{bourbaki} that the set of 
discontinuity points of the map \eqref{mapg0ggg} is of first category, and therefore its set of continuity points is nonempty. Since the map
\eqref{mapg0ggg} is $n$-linear, \cite[Lemma 4.8]{nsfreshetsuper} implies that 
it is continuous.

To complete the proof it is enough to show that the map
\begin{equation}
\label{g1gn-1hrhop}
\g g_\ood\times\g g^{n-1}\to\mathscr H\ ,\ 
(x,x_1,\ldots,x_{n_1})\mapsto 
\rho^\pi(x)\rho^\pi(x_1)\cdots\rho^\pi(x_{n-1})v
\end{equation}
is continuous at $(0,\ldots,0)\in\g g_\ood\times \g g^{n-1}$.
If $(x,x_1,\ldots,x_{n-1})\in \g g_\ood\times\g g^{n-1}$ then
\begin{align*}
\|\rho^\pi(x)&\rho^\pi(x_1)\cdots\rho^\pi(x_{n-1})v\|^2\\
&=
\langle
\rho^\pi(x)\rho^\pi(x_1)\cdots\rho^\pi(x_{n-1})v,
\rho^\pi(x)\rho^\pi(x_1)\cdots\rho^\pi(x_{n-1})v
\rangle\\
&=|\langle
\rho^\pi(x_1)\cdots\rho^\pi(x_{n-1})v,
\rho^\pi(x)^2\rho^\pi(x_1)\cdots\rho^\pi(x_{n-1})v
\rangle|\\
&\leq 
\frac{1}{2}
\|\rho^\pi(x_1)\cdots\rho^\pi(x_{n-1})v\|
\cdot
\|\rho^\pi([x,x])\rho^\pi(x_1)\cdots\rho^\pi(x_{n-1})v\|.
\end{align*}
Therefore continuity of \eqref{g1gn-1hrhop} follows from the 
the induction hypothesis, continuity of \eqref{mapg0ggg}, and continuity of the superbracket of $\g g$.
\end{proof}

\begin{lemma}
\label{smthvecbcmsan}
Let $(G,\g g)$ be an analytic Harish--Chandra pair. Assume that 
$\g g$ is a Fr\' echet--Lie superalgebra.
Let 
$(\pi,\mathscr H,\rho^\pi)$ be a smooth unitary representation of $(G,\g g)$. If $v\in\mathscr H^\omega$ then $\rho^\pi(x)v\in\mathscr H^\omega$ for every $x\in \g g$. 
\end{lemma}
\begin{proof}
The argument appears in the proof of \cite[Thm 6.13]{nsfreshetsuper}, but for the reader's convenience we explain the details.
Since $\mathscr H^\omega$ is $\g g_\eev$--invariant, 
it suffices to prove that $\rho^\pi(x)v\in\mathscr H^\omega$ for every 
$x\in \g g_\ood$. Set $w:=\rho^\pi(x)v$.
By \cite[Thm 5.2]{neebanalytic} it suffices to prove that
the map 
\begin{equation}
\label{gtocgpiww}
G\to\C\ ,\ g\mapsto\langle \pi(g)w,w\rangle
\end{equation}
is analytic. Note that 
$
\langle \pi(g)w,w\rangle
=\sqrt{-1}\langle\pi(g)v,\rho^\pi(g\cdot x)w\rangle
$. Since $w\in\mathscr H^\infty$, Lemma \ref{lme-contofgnh} implies that the linear map $z\mapsto \rho^\pi(z)w$ is continuous, and therefore the map
$g\mapsto \rho^\pi(g\cdot x)w$ is analytic. Since the map 
$g\mapsto \pi(g)v$ is analytic, the map \eqref{gtocgpiww} is also analytic.
\end{proof}

\subsection{The involutive monoid associated to $(G,\g g)$}

\label{sec-invlt}

The anti-linear map \[
\g g_\C\to \g g_\C\ ,\ x\mapsto x^*\] defined by 
\[
x^*:=
\begin{cases}
-x&\text{ if }x\in \g g_\eev,\\
-\sqrt{-1}\,x&\text{ if }x\in\g g_\ood.
\end{cases}
\]
is an anti-automorphism. It 
 extends to an anti-linear anti-automorphism 
\begin{equation}
\label{dfostarrr}
U(\g g_\C)\to  U(\g g_\C)\ ,\ 
D\mapsto D^*
\end{equation}
in a canonical way.
Consider the monoid $\Smi$ with underlying set $G\times U(\g g_\C)$ and multiplication
\[
(g_1,D_1)(g_2,D_2)=(g_1g_2,(g_2^{-1}\cdot D_1)D_2)
\]
where $g\cdot D$ denotes the adjoint action of $g\in G$ on 
$D\in U(\g g_\C)$.
The neutral element of $\Smi$ is $1_\Smi:=(\yek_G,1_{U(\g g_\C)})$.
The map 
\[
\Smi\to \Smi\,,\,s\mapsto s^*
\]
defined by 
\[
(g,D)^*:=(g^{-1},g\cdot(D^*))
\]
is an involution of $\Smi$.

Recall that $U(\g g_\C)$ is an associative superalgebra. 
An element $(g,D)\in \Smi$ is called \emph{odd} 
(resp.  \emph{even}) if $D$ is an odd (resp. even) element of $U(\g g_\C)$. 

\subsection{Smooth and analytic superfunctions on Harish--Chandra pairs}
Similar to Section \ref{thealgrevis}, 
let $\higuc$ (resp., $\higuo$)
denote the set of
$\C$--linear maps 
\[
f:U(\g g_\C)\to C^\infty(G,\C)
\] 
which satisfy the following two properties.
\begin{itemize}
\item[(i)] $f(xD)(g)=\lddd_x(f(D))(g)$ for every $x\in\g g_\eev$, every $D\in U(\g g_\C)$, and every $g\in G$.
\item[(ii)] For every $n\geq 0$ the map
\begin{equation}
f^{[n]}:\g g^n\times G\to \C
\ ,\ f^{[n]}(x_1,\ldots,x_n,g):= 
\big(f(x_1\cdots x_n)\big)(g)
\end{equation}
is smooth (resp., analytic).
\end{itemize}
\begin{remark}
Because of Theorem \ref{cinfandhom},
the spaces 
$\higuc$ and $\higuo$ deserve to be called the 
spaces
of smooth and analytic superfunctions on the Harish--Chandra pair $(G,\g g)$. 
\end{remark}

For every $f\in\higuc$ we set
\[
\Sg{f}:\Smi\to \C\ ,\  
\Sg{f}(g,D):=f(D)(g)
\] 
for every 
$g\in G$ and every $D\in U(\g g_\C)$.

\begin{lemma}
\label{lem-chainruldif}
Let $f\in\higuc$, $(g,D)\in \Smi$, and $x\in\g g_\eev$. 
Then
\[
\lim_{t\to 0}
\frac{1}{t}\left(\Sg{f}(ge^{tx},e^{-tx}\cdot D)
-\Sg{f}(g,D)\right)=\Sg{f}(g,Dx).
\] 
\end{lemma}
\begin{proof}
By linearity of the map $D\mapsto f(D)$ we can assume that $D$ is a monomial of degree $n$. By the Chain Rule we have
\begin{align*}
\lim_{t\to 0}
\frac{1}{t}
\big(\Sg{f}(ge^{tx}&,e^{-tx}\cdot D)
 - \Sg{f}(g,D)\big)\\
&=\frac{d}{dt}\Big|_{t=0}f(e^{-tx}\cdot D)(ge^{tx})
=f(-xD+Dx)(g)+\lddd_x\big(f(D)\big)(g)\\
&=f(-xD+Dx)(g)+f(xD)(g)=f(Dx)(g)=\Sg{f}(g,Dx).
\qedhere
\end{align*}
\end{proof}
Observe that $\Smi$ acts on $\C^\Smi$ (the space of complex-valued functions on $\Smi$) 
by right translation, that is,
$(s\cdot \psi)(t):=\psi(ts)$ for every 
$s,t\in \Smi$ and every $\psi\in \C^\Smi$.
The next lemma shows that $C^\infty(G,\g g)$ is an invariant subspace of $\C^\Smi$ under this action.
\begin{lemma}
\label{lem-fugchb}
Let $f\in \higuc$ and $(g_\circ,D_\circ)\in \Smi$. Then the map 
\[
h:U(\g g_\C)\to C^\infty(G,\C)\ ,\ 
h(D)(g):=f\big((g_\circ^{-1}\cdot D) D_\circ\big)(gg_\circ)
\]
belongs to $\higuc$. 
\end{lemma}
\begin{proof}
The only nontrivial statement is that 
$\lddd_x (h(D))(g)=h(xD)(g)$ for 
$x\in\g g_\eev$ and $D\in U(\g g_\C)$. This can 
be checked as follows:
\begin{align*}
\lddd_x(h(D))(g)
&
=\lim_{t\to 0}\frac{1}{t}\big(h(D)(ge^{tx})-h(D)\big)\\
&=
\lim_{t\to 0}
\frac{1}{t}
\Big(
f\big((g_\circ^{-1}\cdot D)D_\circ\big)(ge^{tx}g_\circ)
-
f\big((g_\circ^{-1}\cdot D)D_\circ\big)(gg_\circ)
\Big)\\
&=
\lim_{t\to 0}
\frac{1}{t}
\Big(
f\big((g_\circ^{-1}\cdot D)D_\circ\big)
(gg_\circ e^{t(g_\circ^{-1}\cdot x)})
-
f\big((g_\circ^{-1}\cdot D)D_\circ\big)(gg_\circ)
\Big)\\
&=\lddd_{g_\circ^{-1}\cdot x} \big(f((g_\circ^{-1}\cdot D)D_\circ)\big)
(gg_\circ)=f\big((g_\circ^{-1}\cdot(xD))D_\circ\big)(gg_\circ)\\
&=h(xD)(g).
\qedhere
\end{align*}

\end{proof}

\subsection{Positive definite smooth superfunctions} We can now define positive definite smooth superfunctions on a Harish--Chandra pair using the involutive monoid 
$\Smi$ introduced in Section \ref{sec-invlt}.
\begin{definition}
\label{def-pdfunc}
An $f\in\higuc$ is called \emph{even} if $\Sg{f}(g,D)=0$ for every odd element $(g,D)\in \Smi$. An $f\in\higuc$ is called 
\emph{positive definite} if $f$ is even and $\Sg{f}$ is a positive definite function on $\Smi$, i.e.,
\[
\sum_{1\leq i,j\leq n}\overline{c_i} c_j 
\Sg{f}(s_i^*s_j)\geq 0
\,\text{ for }n\geq 1,\, 
c_1,\ldots,c_n\in\C,\, 
\text{ and }s_1,\ldots,s_n\in \Smi.
\]

\end{definition}

\subsection{Matrix coefficients of unitary representations} For smooth and analytic unitary representations of a Harish--Chandra pair $(G,\g g)$, the matrix coefficients are defined as follows.

\begin{definition}
Let $(\pi,\mathscr H,\rho^\pi)$ be a (smooth or analytic) unitary representation of $(G,\g g)$. For every 
$v,w\in\mathscr H^\infty$ the function
\[
\varphi_{v,w}:\Smi\to \C\,,\,\varphi_{v,w}(g,D):= \langle\pi(g)\rho^\pi(D)v,w\rangle
\]
is called the \emph{matrix coefficient} of the vectors $v,w$.
\end{definition}

\begin{proposition}
\label{mctoposdef}
Let $(G,\g g)$ be a Harish--Chandra pair such that $\g g$ is a 
Fr\' echet--Lie superalgebra. 
\begin{itemize}
\item[(i)] Let $(\pi,\rho^\pi,\mathscr H)$ be a 
smooth unitary representation of $(G,\g g)$ and
$v,w\in\mathscr H^\infty$ be
homogeneous vectors such that $|v|=|w|$. 
Then there exists an even $f\in\higuc$ 
such that $\Sg{f}=\varphi_{v,w}$.
\item[(ii)] Assume that $(G,\g g)$ is an analytic Harish--Chandra pair. Let $(\pi,\rho^\pi,\mathscr H)$ 
be an analytic unitary representation of $(G,\g g)$ and  $v,w\in\mathscr H^\omega$ be
homogeneous vectors such that $|v|=|w|$. 
Then there exists an even $f\in\higuo$ 
such that 
$\Sg{f}=\varphi_{v,w}$.

\end{itemize}
\end{proposition}
\begin{proof}
(i) It is fairly straightforward to check that the map 
\[
f:U(\g g_\C)\to C^\infty(G,\C)\ ,\ 
f(D)(g):=\varphi_{v,w}(g,D)
\]
is in $\homrm_{\g g_\eev}(U(\g g_\C),C^\infty(G,\C))$. 
To complete the proof, we need to show that for
every $n\geq 0$ the map
\begin{equation}
\label{vpinvwggnc}
\varphi_{n,v,w}:
\g g^n\times G\to \C
\ ,\ \varphi_{n,v,w}(x_1,\ldots,x_n,g):= \langle\pi(g)\rho^\pi(x_1\cdots x_n)v,w\rangle
\end{equation}
is smooth. 
Note that
$
\varphi_{n,v,w}(x_1,\ldots,x_n,g)
=\langle
\rho^\pi(x_1\cdots x_n)v,\pi(g^{-1})w
\rangle.
$
By Lemma~\ref{lme-contofgnh} the $n$-linear map 
$(x_1,\ldots,x_n)\mapsto \rho^\pi(x_1\cdots x_n)v$ is continuous, hence smooth. Since $w\in\mathscr H^\infty$, the map $g\mapsto \pi(g^{-1})w$ is smooth. Therefore the map \eqref{vpinvwggnc} is also smooth.

(ii) The proof is similar to (i), as 
a continuous multilinear map is analytic.
\end{proof}

In Theorem \ref{gns-smth} we prove a converse for Proposition
\ref{mctoposdef}.

\begin{remark}
Assume $\g g_\ood=\{0\}$. Then the condition of Definition \ref{def-pdfunc} a priori seems to be stronger than the condition which is classically used to define positive definite functions on a Lie 
group: classically, a map $f:G\to\C$ is called positive definite if  
\[
\sum_{1\leq i,j\leq n}c_i\overline c_jf(g_ig_j^{-1})\geq 0
\ \text{ for }\,n\geq 1,\, g_1,\ldots, g_n\in G\,, 
\text{ and }c_1,\ldots,c_n\in \C.
\] 
However, for smooth maps $f:G\to\C$, the classical definition and Definition \ref{def-pdfunc} are equivalent. In fact, if $f\in C^\infty(G,\C)$ is positive definite in the classical sense, 
then by the  GNS construction \cite[III.1.22]{neebbook}
we have
$f(g)=\langle \pi(g)v,v\rangle$ for some unitary representation $(\pi,\mathscr H)$ of $G$. Since $f$ is smooth, from \cite[Thm 7.2]{neebdiff} it follows that $v\in\mathscr H^\infty$, and therefore the map 
\[
\Smi\to \C\ ,\ (g,D)\mapsto \langle
\pi(g)\rho^\pi(D)v,v\rangle
\]
is well-defined. It is easy to check that the latter map is 
positive definite in the sense of
Definition \ref{def-pdfunc}.
\end{remark}

\subsection{The reproducing kernel Hilbert space}

To every smooth unitary representation
$(\pi,\rho^\pi,\mathscr H)$ of $(G,\g g)$ 
we can associate a representation
$\widetilde{\rho^\pi}$ of the monoid $\Smi$ as follows:
\begin{equation}
\label{wtrhopii}
\widetilde{\rho^\pi}:\Smi\to\mathrm{End}(\mathscr H^\infty)\ ,\
\widetilde{\rho^\pi}(g,D):=\pi(g)\rho^\pi(D)\text{ for every }
(g,D)\in \Smi.
\end{equation}
Observe that $(\widetilde{\rho^\pi},\mathscr H^\infty)$  is a 
\emph{$*$-representation}, i.e., 
$\langle \widetilde{\rho^\pi}(s)v,w\rangle=\langle v,\widetilde{\rho^\pi}(s^*)w\rangle$ for every $s\in \Smi$ and every $v,w\in \mathscr H^\infty$.
It is easy to check that for every $v\in\mathscr H^\infty$ the matrix coefficient
$\varphi_{v,v}$
is positive definite. 

Conversely, one can associate a $*$-representation of $\Smi$ to a positive definite function $\varphi:\Smi\to\C$ as follows. 
Set 
\begin{equation}
\label{eq-Dsubf}
\mathscr D_\varphi:=
\mathrm{Span}_\C \{\varphi_s\ :\ s\in \Smi\}
\subseteq\C^\Smi
\end{equation}
where $\varphi_s:\Smi\to \C$ is defined by $\varphi_s(t):=\varphi(ts)$. 
Observe that $\mathscr D_\varphi$ has a pre-Hilbert space structure given by
\[
\langle \varphi_s,\varphi_t\rangle:=\varphi(t^*s)
.\]
Set $K(s,t):=\varphi(st^*)$ and define $K_s:\Smi\to \C$ by
$K_s(t):=K(t,s)=\varphi_{s^*}(t)$.
The completion 
$\mathscr H_\varphi$ of $\mathscr D_\varphi$ 
is a reproducing kernel Hilbert space with kernel
$K(s,t)$. In other words, one can identify $\mathscr H_\varphi$ 
with a space of complex valued functions on $\Smi$ such that 
\begin{equation}
\label{eq-hs=hks}
h(s)=\langle h,K_s\rangle \text{ for every }h\in\mathscr H_\varphi\text{ and every }s\in \Smi.
\end{equation}

The monoid $\Smi$ acts on $\mathscr D_\varphi$ by right translation, yielding a $*$-representation 
$(\widetilde{\rho_\varphi},\mathscr D_\varphi)$ of $\Smi$. More precisely,
\[
\big(\widetilde{\rho_\varphi}(s)\psi\big)(t)=\psi(ts)\text{ for every }s,t\in \Smi\text{ and every }\psi\in\mathscr D_\varphi.
\]
\begin{remark}
\label{rem-unitabsgrp}
If an element $s\in \Smi$ satisfies $ss^*=s^*s=1_\Smi$, 
then $\widetilde{\rho_\varphi}(s):\mathscr D_\varphi\to\mathscr D_\varphi$  is an isometry and therefore extends to an isometry 
$\widetilde{\rho_\varphi}(s):\mathscr H_\varphi\to\mathscr H_\varphi$, yielding a unitary representation of the abstract group $\{s\in \Smi\ :\ ss^*=s^*s=1_\Smi\}$.
\end{remark}

\subsection{Cyclic representations and the GNS construction}
Our next goal is to prove Theorem \ref{gns-smth} below, which is the analogue of the GNS construction for Lie supergroups.

\begin{definition}
A smooth 
(resp. analytic) unitary representation $(\pi,\rho^\pi,\mathscr H)$ is called \emph{cyclic} if there exists a vector $v_\circ\in\mathscr H^\infty_\eev$ (resp. $v_\circ\in\mathscr H^\omega_\eev$)
such that the set
\[
\mathrm{Span}_\C\{\pi(g)\rho^\pi(D)v_\circ\ :\ g\in G\text{ and } D\in U(\g g_\C)\}
\] 
is dense in $\mathscr H$. The vector $v_\circ$ is called a \emph{cyclic vector} of $(\pi,\rho^\pi,\mathscr H)$. 
\end{definition}
\begin{remark}
To indicate that a unitary representation $(\pi,\rho^\pi,\mathscr H)$ is cyclic with a cyclic vector $v_\circ$, we write $(\pi,\rho^\pi,\mathscr H,v_\circ)$.
\end{remark}
\begin{definition}
We say a Lie group $G$ has the \emph{Trotter property} if for every
$x,y\in\Lie(G)$ 
the equality
\[
e^{t(x+y)}=\lim_{n\to\infty}
\big(e^{\frac{t}{n}x}e^{\frac{t}{n}y}\big)^n
\]
holds in the sense of uniform convergence on compact subsets of $\R$.
\end{definition}
\begin{example}
\label{ex:trot}
We now 
mention some examples of Lie supergroups $\Gsu$ for which $\Gsu_{\Lambda_0}$ has the Trotter property.
Proofs are given in \cite{nsfreshetsuper}.

\begin{itemize}
\item[(i)]
Every locally exponential Lie group (and 
in particular every Banach--Lie group) has the Trotter property. Therefore if $\Gsu$ is a Lie supergroup modeled on a $\Z_2$--graded Banach space, then 
$\Gsu_{\Lambda_0}$ has the Trotter property.
\item[(ii)] From (i) it follows that 
if $M$ is a compact smooth manifold and $K$ is a 
finite dimensional Lie group,  
then the mapping group $C^\infty(M,K)$ and
its
central extensions 
have the Trotter property. These Lie groups appear as  
$\Gsu_{\Lambda_0}$ of mapping Lie supergroups  
$\Gsu:=C^\infty(M,\mathcal K)$, where $M$ is a compact manifold and $\mathcal K$ is a finite-dimensional Lie supergroup, or 
mapping Lie
supergroups $\Gsu:=C^\infty(\Msu,K)$, where $\Msu$ is a compact supermanifold and
$K$ is a finite-dimensional Lie group.

\item[(iii)] If $M$ is a compact smooth manifold then the group $\mathrm{Diff}(M)$ of smooth diffeomorphisms of $M$  and its central extensions have 
the Trotter property. This implies that 
if $\Gsu$ is the Lie supergroup of smooth diffeomorphisms of certain compact supermanifolds, such as the supercircle $S^{1|1}$, then $\Gsu_{\Lambda_0}$ has the Trotter property. We expect the latter statement to hold for other classes of compact supermanifolds.

\end{itemize}
\end{example}

\begin{theorem}
\label{gns-smth}
Let $(G,\g g)$ be a Harish--Chandra pair such that $\g g$ is a Fr\' echet--Lie superalgebra and $G$ has the Trotter property. 
Let $f\in \higuc$ be
positive definite. 
\begin{itemize}
\item[(i)]  
There exists a cyclic smooth unitary representation 
$(\pi,\rho^{\pi},\mathscr H,v_\circ)$ of 
$(G,\g g)$ such that
$\Sg{f}=\varphi_{v_\circ,v_\circ}$.
\item[(ii)] Let $(\sigma,\rho^{\sigma},\mathscr K,w_\circ)$ be another cyclic smooth unitary representation of $(G,\g g)$ such that $\Sg{f}=\varphi_{w_\circ,w_\circ}$.
Then 
$(\pi,\rho^{\pi},\mathscr H)$ and $(\sigma,\rho^{\sigma},\mathscr K)$ are unitarily equivalent via 
an intertwining operator
\[
T:(\sigma,\rho^{\sigma},\mathscr K)\to
(\pi,\rho^{\pi},\mathscr H)
\] 
which maps $w_\circ$ to $v_\circ$.

\item[(iii)] Assume that $(G,\g g)$ is an analytic Harish--Chandra pair and 
the map 
\[
G\to \C\ ,\ g\mapsto f(1_{U(\g g_\C)})(g)
\]
is analytic. Then the representation $(\pi,\mathscr H,\rho^\pi,v_\circ)$ obtained in {\upshape(i)} is an 
analytic representation of $(G,\g g)$.
\end{itemize}

\end{theorem}
\begin{proof}
(i)
Set $\varphi:=\Sg{f}$. Let $\mathscr H:=\mathscr H_\varphi$ be the Hilbert completion of $\mathscr D_\varphi$  defined in
\eqref{eq-Dsubf}, and $v_\circ:=\varphi\in\mathscr H$.
For every 
$s:=(g,D)\in \Smi=G\times U(\g g)$ we have
\[
\varphi_{v_\circ,v_\circ}(s)=\langle\varphi_s,\varphi\rangle=\varphi(s)=\Sg{f}(s).
\]
By Remark \ref{rem-unitabsgrp}, 
if $g\in G$ then 
$\widetilde{\rho_\varphi}\big(g,1_{U(\g g_\C)}\big)$ extends to an isometry of $\mathscr H_\varphi$.
Setting $\pi(g):=\widetilde{\rho_\varphi}
\big(g,1_{U(\g g_\C)}\big)$
we obtain a representation $\pi$ of $G$ on $\mathscr H$ by unitary operators.

Our next goal is to prove that $(\pi,\mathscr H)$ is a smooth unitary representation of $G$ and $\mathscr D_\varphi\subseteq \mathscr H^\infty$, where $\mathscr H^\infty$ is the subspace of smooth vectors of $(\pi,\mathscr H)$. By \cite[Theorem 7.2]{neebdiff} it suffices to prove that 
for every $s\in \Smi$ the map
\begin{equation}
\label{gccgpiphi}
G\to \C\ ,\ g\mapsto \langle\pi(g)\varphi_s,\varphi_s\rangle
\end{equation}
is smooth. Fix 
$s=(g_\circ,D_\circ)\in \Smi$ and let 
$\tilde g:=(g,1_{U(\g g_\C)})\in \Smi$ for every $g\in G$. Note that
\[
\langle\pi(g)\varphi_s,\varphi_s\rangle
=\langle\varphi_{\tilde gs},\varphi_s\rangle
=\varphi(s^*\tilde gs)=f\big((g_\circ^{-1}g^{-1}g_\circ\cdot D_\circ^*) D_\circ\big)(g_\circ^{-1}gg_\circ)
\]
and therefore 
smoothness of \eqref{gccgpiphi} follows from
smoothness of the map \eqref{ggGsmuth}.

Next we prove that if $x\in \g g_\eev$ and  $v\in\mathscr D_\varphi$ then $\dd\pi(x)v=\widetilde{\rho_\varphi}\big(\yek_G,x\big)v$. It suffices to take $v=\varphi_s$ for some $s:=(g_\circ,D_\circ)\in \Smi$. 
Let $s':=(g,D)\in \Smi$. Then
$
\varphi_s(s')=\varphi(s's)=f\big((g_\circ^{-1}\cdot D)D_\circ\big)(gg_\circ)
$.
Lemma \ref{lem-fugchb} implies that $\varphi_s=\Sg{h}$ for some $h\in\higuc$. 
From \eqref{eq-hs=hks} and 
Lemma \ref{lem-chainruldif} it follows that
\begin{align*}
\big(\dd\pi(x)\varphi_s\big)(s')&=\langle \dd\pi(x)\Sg{h},K_{s'}\rangle=
\langle\lim_{t\to 0}
\frac{1}{t}(\pi(e^{tx})\Sg{h}-\Sg{h}),K_{s'}\rangle\\
&=\lim_{t\to 0}\frac{1}{t}
\langle\pi(e^{tx})\Sg{h}-\Sg{h}
,K_{s'}\rangle=
\lim_{t\to 0}\frac{1}{t}
(\Sg{h}(ge^{tx},e^{-tx}\cdot D)-\Sg{h}(g,D))\\
&=\Sg{h}(g,Dx)=\varphi_s(g,Dx)=\varphi_s
\big((g,D)(\yek_G,x)\big)=
\big(\widetilde{\rho_\varphi}\big(\yek_G,x\big)
\varphi_s\big)(s').
\end{align*}

Finally, to complete the proof of existence of 
$(\pi, \rho^\pi,\mathscr H,v_\circ)$ set
$
\rho^{\mathscr D_\varphi}(x)
:=
\widetilde{\rho_\varphi}(\yek_G,x)
$ 
for every 
$x\in\g g$.
From linearity of 
$
\varphi_s(g,D)
$
in $D$
it follows directly that the $\Z_2$-grading of $U(\g g_\C)$ induces a $\Z_2$--grading on $\mathscr D_\varphi$ (and hence on 
$\mathscr H_\varphi$) and the actions of $G$ and $\g g$ are compatible with the $\Z_2$--grading.
From \cite[Lem. 2.2(a)]{menesa} 
it follows that for every $x\in\g g_\eev$,
the operator
 $\rho^{\mathscr D_\varphi}(x)$ is 
essentially skew-adjoint.
Consequently,
the 4-tuple 
$(\pi,\mathscr H, \mathscr D_\varphi,\rho^{\mathscr D_\varphi})$ is 
a pre-representation of $(\g g,G)$ in the sense of \cite[Def. 6.4]{nsfreshetsuper}. Therefore the existence of $(\pi,\rho^\pi,\mathscr H)$ follows from 
\cite[Thm 6.13]{nsfreshetsuper}.

(ii) In principle, the proof is similar to the standard uniqueness proofs of the GNS construction (e.g., see \cite[Thm III.1.22]{neebbook}). 
Nevertheless,  \cite[Thm III.1.22]{neebbook} 
is not directly applicable because for example the representation of the semigroup $\Smi$ is not bounded.
The new technical issues that arise in the super context 
will be addressed below.

Let $\mathscr H^\infty$ (resp., $\mathscr K^\infty$) denote the set of smooth vectors of $(\pi,\mathscr H)$ (resp., $(\sigma,\mathscr K)$). As in
\eqref{wtrhopii} we obtain $*$-representations 
$(\widetilde{\rho^\pi},\mathscr H^\infty)$ and $(\widetilde{\rho^\sigma},\mathscr K^\infty)$ of $\Smi$. Set 
\[
\mathscr D_{v_\circ}:=\mathrm{Span}_\C\{\widetilde{\rho^\pi}(s)v_\circ\ :\ s\in \Smi\}
\text{ and }
\mathscr D_{w_\circ}:=\mathrm{Span}_\C\{\widetilde{\rho^\sigma}(s)w_\circ\ :\ s\in \Smi\}.
\]
Define a $\C$-linear map 
$T:\mathscr D_{w_\circ}\to\mathscr D_{v_\circ}$ by 
\[
T\widetilde{\rho^\sigma}(s)w_\circ:=
\widetilde{\rho^\pi}(s)v_\circ 
\text{ for all }s\in \Smi.
\] 
It is straightforward to check that $T$ is well-defined and extends to an isometry $T:\mathscr K\to\mathscr H$. 
From the definition of $T$ it follows that 
\begin{equation}
\label{partialinter}
T\widetilde{\rho^\sigma}(s)u=\widetilde{\rho^\pi}(s)Tu
\text{ for every }u\in \mathscr D_{w_\circ}\text{ and every }s\in \Smi.
\end{equation}
Since $\mathscr D_{w_\circ}$ 
is dense in $\mathscr K$, from \eqref{partialinter} it follows that 
\[
T\sigma(g)u=\pi(g)Tu\text{ for every }u\in\mathscr K\text{ and every }
g\in G.
\]
Consequently, $T\mathscr K^\infty=\mathscr H^\infty$. Next we prove that
\[
T\rho^\sigma(x)u=\rho^\pi(x)Tu\text{ for every }u\in \mathscr K^\infty\text{ and every }x\in\g g.
\]
It suffices to prove the latter statement for $x\in\g g_\ood$. Set 
\[
P_1:=e^{-\frac{\pi i}{4}}\rho^\pi(x),\,
P_2:=e^{-\frac{\pi i}{4}}T\rho^\sigma(x)T^{-1}\res{\mathscr H^\infty}
\text{ and }\mathscr L:=T\mathscr D_{w_\circ}=\mathscr D_{v_\circ}.
\]
The linear operators $P_1$ and $P_2$ are symmetric
 with common domain $\mathscr H^\infty$ such that 
\[
P_1\res{\mathscr L}=P_2\res{\mathscr L}.
\]
Since $\mathscr L$ is $G$--invariant, from \cite[Lem. 2.2(a)]{menesa}
it follows that $(P_1\res{\mathscr L})^2$ is essentially self-adjoint, and
therefore by \cite[Lem. 2.4]{menesa} the operator
$P_1\res{\mathscr L}$ is essentially self-adjoint.
Consequently, by  \cite[Lem. 2.5]{menesa} we have $P_1=P_2$.

(iii) Let $(\pi,\rho^\pi,\mathscr H)$ be the smooth unitary representation obtained in (i). 
By \cite[Thm 6.13(b)]{nsfreshetsuper} it is enough to 
show that $\mathscr D_\varphi\subseteq\mathscr H^\omega$.  Since 
$
\langle \pi(g)\varphi,\varphi\rangle
=f(1_{U(\g g_\C)})(g)
$,
from \cite[Thm 5.2]{neebanalytic} 
it follows that $\varphi\in\mathscr H^\omega$. 
From Lemma~\ref{smthvecbcmsan} it follows that $\rho^\pi(D_\circ)\varphi\in\mathscr H^\omega$ for every $D_\circ\in U(\g g_\C)$. Finally, for every
$s:=(g_\circ,D_\circ)\in \Smi$ we have
\[
\varphi_s=\widetilde{\rho_\varphi}(s)\varphi=
\widetilde{\rho_\varphi}(g_\circ,1_{U(\g g_\C)})\widetilde{\rho_\varphi}(\yek_G,D_\circ)\varphi=\pi(g_\circ)
\rho^\pi(D_\circ)\varphi\in\mathscr H^\omega
\]
because $\mathscr H^\omega$ is $G$--invariant.
\end{proof}

The next corollary, which is of independent interest, is in a sense an automatic analyticity criterion for smooth 
superfunctions in odd directions.

\begin{corollary}
\label{cor-autoan}
Let $\Gsu$ be an analytic Lie supergroup modeled on a $\Z_2$--graded Fr\' echet space and $(G,\g g)$ be the Harish--Chandra pair associated to $\Gsu$. 
Let $f\in C^\infty(G,\g g)$ be positive definite. If 
Assume that $G$ has the Trotter property. 
If $f(1_{U(\g g_\C)})\in C^\omega(G,\C)$ then $f\in C^\omega(G,\g g)$. 

\end{corollary}

\begin{proof}
From Theorem \ref{gns-smth}(iii) it follows that $\Sg{f}=\varphi_{v,v}$, where $v$ is an analytic vector.
Thus, by Proposition \ref{mctoposdef},
we have $f\in\higuo$.
\end{proof}

\section{A characterisation of integrable linear functionals}
Let $(G,\g g)$ be an analytic Harish--Chandra pair such that $G$ is 1-connected (that is, connected and simply connected).
In this section we give a characterisation of 
$\C$--linear functionals $\lambda:U(\g g_\C)\to\C$ which are 
\emph{integrable} in the sense that
there exists an analytic unitary representation $(\pi,\rho^\pi,\mathscr H)$  
of $(G,\g g)$ such that $\lambda(D)=\langle \rho^\pi(D)v,v\rangle$ for some $v\in\mathscr H^\omega$. For Lie groups, this question is addressed in detail in \cite[Sec. 6]{neebanalytic}. 


\subsection{Weak and strong analyticity of linear functionals}
We begin by defining the notion of analyticity of a linear functional on $U(\g g_\C)$.
\begin{definition}
Let $\lambda:U(\g g_\C)\to \C$ be a $\C$--linear map.
We say $\lambda$ is \emph{even} if 
$\lambda(D)=0$ for every $D\in U(\g g_\C)_\ood$.
We say $\lambda$ is \emph{positive} if 
$\lambda$ is even and $\lambda(D^*D)\geq 0$ for every $D\in U(\g g_\C)$, where $D\mapsto D^*$ is the map 
given in
\eqref{dfostarrr}.
We say $\lambda$ is 
\emph{continuous} if for every $n\geq 0$ the map
\[
\g g^n\to \C\ ,\ 
(x_1,\ldots,x_n)\mapsto \lambda(x_1\cdots x_n)
\]
is continuous.
A continuous $\C$--linear map 
$\lambda:U(\g g_\C)\to \C$
 is called \emph{weakly analytic} if 
for every $D_1,D_2\in U(\g g_\C)$ there exists a
0-neighborhood $U_{D_1,D_2}\subseteq \g g_{\C,\eev}:=\g g_\eev\otimes_\R\C$ such that the series 
\[
\sum_{n=0}^\infty \frac{1}{n!}|\lambda(D_1x^nD_2)|
\]
converges for every $x\in U_{D_1,D_2}$.
A continuous $\C$--linear map 
$\lambda:U(\g g_\C)\to \C$
 is called  \emph{strongly analytic} if there exists a
0-neighborhood $U\subseteq\g g_{\C,\eev}$ 
such that the series
\[
\sum_{n=0}^\infty \frac{1}{n!}|\lambda(D_1x^nD_2)|
\]
converges for every $D_1,D_2\in U(\g g_\C)$ and every $x\in U$.
\end{definition}

\subsection{Characterisations of integrability}
We recall the definition of a BCH--Lie group.

\begin{definition}
An analytic  Lie group $G$ is called a \emph{BCH--Lie group} if the exponential map is an analytic diffeomorphism in an open 0-neighborhood.
\end{definition}
Every Banach--Lie group is a BCH Lie group. 
For a detailed study 
and several interesting examples of
BCH--Lie
 groups  
 see
\cite{glockner} and \cite[Sec. IV]{neebjpn}.

\begin{theorem}
\label{thm-cinteg}
Let $(G,\g g)$ be an analytic Harish--Chandra pair
such that $\g g$ is a Fr\' echet--Lie superalgebra and $G$ is a 1-connected BCH--Lie group. 
Let
\[\lambda:U(\g g_\C)\to\C
\] be a
$\C$--linear map.
The  
following statements are equivalent.
\begin{itemize}
\item[(i)] $\lambda$ is positive and strongly analytic.

\item[(ii)] There exists an analytic unitary representation $(\pi,\rho^\pi,\mathscr H)$ of $(G,\g g)$ and a homogeneous 
vector $v\in\mathscr H^\omega$ such that
$\lambda(D)=\langle \rho^\pi(D)v,v\rangle$ for every $D\in U(\g g_\C)$.

\end{itemize}

\end{theorem}

\begin{proof}
(i)$\Rightarrow$(ii): 
Let $U(\g g_\C)^*$ denote the algebraic 
dual of 
$U(\g g_\C)$ and 
\[
\rho:\g g_\C\to\End_\C(U(\g g_\C)^*)
\ ,\ 
\big(\rho(x)\mu\big)(D):=\mu(Dx)\text{\, for every }
D\in U(\g g_\C)
\]
be the right regular representation of $\g g_\C$ on $U(\g g_\C)^*$.
Set $\mathscr D_\lambda:=\rho(U(\g g_\C))\lambda$.
We denote the restriction of $\rho$ to $\mathscr D_\lambda$ by
$(\rho_\lambda,\mathscr D_\lambda)$.
We endow $\mathscr D_\lambda$ with a pre-Hilbert structure as follows:
\[
\langle \rho_\lambda(D_1)\lambda,\rho_\lambda(D_2)\lambda\rangle:=
\lambda(D_2^*D_1)\text{ for every }D_1,D_2\in U(\g g_\C).
\]
It is easily checked that $\langle \rho_\lambda(D)\mu_1,\mu_2\rangle=
\langle \mu_1,\rho_\lambda(D^*)\mu_2\rangle$ for every $\mu_1,\mu_2\in\mathscr D_\lambda$, i.e., $(\rho_\lambda,\mathscr D_\lambda)$ is a $*$-representation. Since $\lambda$ is even, the $\Z_2$--grading 
of $U(\g g_\C)$ induces a $\Z_2$--grading on $\mathscr D_\lambda$ with perpendicular homogeneous components. 
The rest of the proof is given in the following four steps:

\textbf{Step 1.} 
Since $\lambda$ is continuous, 
$(\rho_\lambda,\mathscr D_\lambda)$ is 
a strongly continuous representation of $\g g_\eev$, 
i.e., for every $\mu\in\mathscr D_\lambda$ the map 
\[
\g g_\eev\to\mathscr D_\lambda\ ,\ 
x\mapsto \rho_\lambda(x)\mu
\]
is continuous.

\textbf{Step 2.}
We claim that 
$\mathscr D_\lambda$ is equianalytic in the sense of 
\cite[Def. 6.4]{neebanalytic}, that is, there exists a 0-neighborhood $V\subseteq \g g_\eev$ such that for every
$\mu\in\mathscr D_\lambda$
the series
\[
\sum_{n=0}^\infty\frac{1}{n!}\|\rho_\lambda(x)^n\mu\|
\] 
converges for all 
$x\in V$. 
Let $U\subseteq \g g_\eev$ be a 0-neighborhood 
such that
the series
\[
\sum_{n=0}^\infty\frac{1}{n!}|\lambda(D_1x^nD_2)|
\]
converges for every $D_1,D_2\in U(\g g_\C)$ and every $x\in U$.
Assume that $\mu=\rho_\lambda(D_\circ)\lambda$ for some $D_\circ\in U(\g g_\C)$. 
Then \[
||\rho_\lambda(x)^n\mu||^2=\langle \rho_\lambda(x^nD_\circ)\lambda, 
\rho_\lambda(x^nD_\circ)\lambda\rangle=|\lambda(D_\circ^*x^{2n}D_\circ)|.
\]
Set $V=\frac{1}{r}U:=\{x\in \g g_\eev\,:\, rx\in U\}$ where $r>2$. If $x\in V$ 
then by the Cauchy--Schwarz inequality
\begin{align*}
\sum_{n=0}^\infty\frac{1}{n!}\|\rho_\lambda(x)^n\mu\|
&=
\sum_{n=0}^\infty\frac{1}{n!}
|\lambda(D_\circ^*x^{2n}D_\circ)|^\frac{1}{2}\\
&\leq
\left(\sum_{n=0}^\infty
\frac{1}{(2n)!}
|\lambda(D_\circ^*(rx)^{2n}D_\circ)|\right)^\frac{1}{2}
\left(
\sum_{n=0}^\infty
\frac{(2n)!}{n!n!r^{2n}}
\right)^\frac{1}{2}<\infty.
\end{align*}


\textbf{Step 3.} 
By \cite[Thm 6.8]{neebanalytic}
there exists a unitary representation $(\pi_\lambda,\mathscr H_\lambda)$ of $G$, where $\mathscr H_\lambda$ is the
completion   of $\mathscr D_\lambda$, such that 
$\dd\pi_\lambda(x)\res{\mathscr D_\lambda}=\rho_\lambda(x)$ for every $x\in\g g_\eev$.

\textbf{Step 4.} From the previous steps and 
\cite[Lem. 2.2(b)]{menesa}
it follows that $(\pi_\lambda,\mathscr H_\lambda,\mathscr D_\lambda,\rho_\lambda)$ is a pre-representation of $(G,\g g)$. By 
\cite[Thm 6.13(b)]{nsfreshetsuper}, the latter pre-representation 
corresponds to an analytic unitary representation of $(G,\g g)$.

(ii)$\Rightarrow$(i): 
To check that $\lambda$ is positive is routine.
Next we prove that $\lambda$ is strongly analytic.
Let $\mathsf P$ denote the set of seminorms that 
define the topology of
$\g g_{\C,\eev}$. For every $p\in \mathsf P$
and every $r>0$ we set 
\[
U_{p,r}:=\{x\in \g g_{\C,\eev}\ :\ p(x)<r\}
\]
and
\[
\mathscr H^{\omega,p,r}:=\left\{w\in\mathscr H^\infty
\ :\ 
\sum_{n=0}^\infty \frac{1}{n!}\dd\pi(x)^nw\text{ converges for every }x\in U_{p,r}
\right\}.
\]
Choose $p_\circ\in\mathsf P$ 
and  $r_\circ>0$ such that the
restriction of the exponential map of $G$ to $U_{p_\circ,r_\circ}\cap \g g_\eev$ is an analytic diffeomorphism, the Baker--Campbell--Hausdorff product defines an analytic function 
$
U_{p_\circ,r_\circ}\times U_{p_\circ,r_\circ}\to\g g_{\C,\eev}
$, and the map 
\[
\g g_\eev\times\g g\to\g g
\ ,\ 
(x,y)\mapsto \Ad(e^x)(y)
\]
extends to a complex analytic map 
\begin{align}
\label{up0r0zgc}
U_{p_\circ,r_\circ}\times \g g_\C\to\g g_\C.
\end{align}
If $p\in\mathsf P$  then we write $p\geq p_\circ$ if $U_{p,r}\subseteq U_{p_\circ,r}$ for some (equivalently, every) $r>0$.
The rest of the proof is given in the following four steps:

\textbf{Step 5.} Let $p\geq p_\circ$, $0<r<r_\circ$,
and $v\in\mathscr H^\infty$. Then 
$v\in\mathscr H^{\omega,p,r}$ if and only if the orbit map
\[
\g g_\eev\mapsto \mathscr H\ ,\ 
x\mapsto \pi(e^x)v
\]
extends to an analytic function on $U_{p,r}$. The proof of the latter statement is similar to the proof 
of \cite[Lem. 3.3]{menesa}. (Here the main point is that $U_{p,r}$ is a balanced 0-neighborhood.)

\textbf{Step 6.} Let $p\geq p_\circ$, $0<r<r_\circ$,
$v\in\mathscr H^{\omega,p,r}$, and $a\in\g g_\eev$.
Then 
\[
\sum_{n=0}^\infty\frac{1}{n!}\dd\pi(x)^nv\in\mathscr H^\omega
\,\text{ for every }
x\in U_{p,r}
\]
and the map 
\[
u_a:U_{p,r}\to\mathscr H\ ,\ 
u_a(x):=\dd\pi(a)\left(\sum_{n=0}^\infty\frac{1}{n!}\dd\pi(x)^nv
\right)
\]
is an analytic function. The latter statement is an extension of \cite[Lem. 3.4]{menesa} to 
Fr\' echet--Lie groups. The proof of
\cite[Lem. 3.4]{menesa} is still valid because
we are assuming that the BCH product formula locally 
defines an analytic function.

\textbf{Step 7.} If $0<r<r_\circ$ and $p\geq p_\circ$, 
then 
$\mathscr H^{\omega,p,r}$ is $\g g_\C$--invariant.
The latter statement is an extension of 
\cite[Prop. 4.9]{menesa} and its proof is an adaptation of the proof of \cite[Prop. 4.9]{menesa}. For the reader's convenience we briefly explain the necessary modifications. Instead of \cite[Lem. 3.3]{menesa}
and \cite[Lem 3.4]{menesa} one uses Steps 5--6 above.
In order to prove that  the map given in 
\cite[Eq. (30)]{menesa} is analytic, one can substitute the norm estimates given in \cite{menesa} by the analyticity of the map \eqref{up0r0zgc}.

\textbf{Step 8.} Choose $0<r<r_\circ$ and $p\geq p_\circ$ such that 
$v\in \mathscr H^{\omega,p,r}$. 
By $\g g_\C$--invariance of $\mathscr H^{\omega,p,r}$ 
the series
\[
\sum_{n=0}^\infty \frac{1}{n!}\rho^\pi(x)^n\rho^\pi(D_2)v
\]
converges for every $x\in U_{p,r}$ and every $D_2\in U(\g g_\C)$. Therefore the series
\[
\sum_{n=0}^\infty\frac{1}{n!}|\lambda(D_1x^nD_2)|
=
\sum_{n=0}^\infty\frac{1}{n!}
|\langle
\rho^\pi(x)^n\rho^\pi(D_2)v,\rho^\pi(D_1^*)v\rangle|
\]
converges for every $D_1,D_2\in U(\g g_\C)$ and every $x\in U_{p,\frac{r}{2}}$. 
\end{proof}

\begin{corollary}
Let $(G,\g g)$ be an analytic Harish--Chandra pair such that $\g g$ is a Banach--Lie superalgebra and $G$ is 1-connected.
Let $\lambda:U(\g g_\C)\to\C$ be a
$\C$--linear map.
The  
following statements are equivalent.
\begin{itemize}
\item[(i)] $\lambda$ is positive and weakly analytic.
\item[(ii)] $\lambda$ is positive and strongly analytic.
\item[(iii)] There exists an analytic unitary representation $(\pi,\rho^\pi,\mathscr H)$ of $(G,\g g)$ and a homogeneous 
vector $v\in\mathscr H^\omega$ such that
$\lambda(D)=\langle \rho^\pi(D)v,v\rangle$ for every $D\in U(\g g_\C)$.

\end{itemize}

\end{corollary}

\begin{proof}
(iii)$\Rightarrow$(ii) follows from Theorem 
\ref{thm-cinteg} and (ii)$\Rightarrow$(i) is trivial. For  (i)$\Rightarrow$(iii) the proof of Theorem \ref{thm-cinteg} still remains valid, because 
for 
Banach--Lie groups
the conclusion of 
\cite[Thm 6.8]{neebanalytic} remains 
true  without assuming equianalyticity.
\end{proof}


\begin{thebibliography}{ZZZZZZ}



\bibitem[AlLa12]{aldlau} Alldridge, A., Laubinger, M., \emph{Infinite-dimensional supermanifolds over arbitrary base fields,} 
Forum Mathematicum,  24 (2012), no. 3, 565--608.


\bibitem[BoSi71]{bochnaksiciak} Bochnak, J., Siciak, J.,
\emph{Analytic functions in topological vector spaces},
Studia mathematica 39 (1971), 77--112.

\bibitem[BFD86]{boucheretal} Boucher, W., Friedan, D., Kent, A., {\it 
Determinant formulae and unitarity for the N =
2 superconformal algebras in two dimensions or exact results on string compactifica-
tion},  Phys. Lett. B 172, 316–322 (1986)



\bibitem[Bo74]{bourbaki} Bourbaki, N., \emph{Topologie G\' en\' erale}, 
Chap. 5 \` a 10, Hermann, 1974.


\bibitem[CCTV]{varadarajan} Carmeli, C., Cassinelli, G., Toigo, A., Varadarajan, V. S., \emph{Unitary representations of super Lie groups and applications to the classification and multiplet structure of super particles}, 
Comm. Math. Phys. 263 (2006), no. 1, 217--258. 






\bibitem[DeMo99]{delignemorgan} Deligne, P., Morgan, J. W., \emph{Notes on supersymmetry (following Joseph Bernstein)}, Quantum fields and strings: a course for mathematicians, Vols. 1, 2 (Princeton, NJ, 1996/1997), 41--97, Amer. Math. Soc., Providence, RI, 1999.




\bibitem[FQS85]{friedanqiu}  Friedan, D., Qiu, Z., Shenker, S.,
\emph{Superconformal invariance in two dimensions and the tricritical Ising model},
Phys. Lett. B 151 (1985), no. 1, 37--43. 





\bibitem[GKO86]{goddard} 
Goddard, P., Kent, A., Olive, D., 
\emph{Unitary representations of the Virasoro and super-Virasoro algebras}, 
Comm. Math. Phys. 103 (1986), no. 1, 105--119.



\bibitem[GeMa88]{gelfandmanin} Gelfand, I., Manin, Y., \emph{Methods of homological algebra},
Second edition. Springer Monographs in Mathematics. Springer-Verlag, Berlin, 2003.




\bibitem[Gl02]{glockner} Gl\" ockner, H., 
\emph{Lie group structures on quotient groups and universal complexifications for infinite-dimensional Lie groups}, Journal of Functional Analysis 194 (2002), 347--409.


\bibitem[GlNe]{glocknerneeb} Gl\" ockner, H., Neeb, K.--H.,
\emph{Infinite-dimensional Lie groups, general theory and main examples},
book in preparation.

\bibitem[Ha82]{hamilton} Hamilton, R., 
\emph{The inverse function theorem of Nash and Moser},
Bulletin of the American Mathematical Society 7 (1982), no. 1, 65--222.


\bibitem[Io08]{iohara8} Iohara, K.: Modules de plus haut poids unitarisables sur la super-alg\`ebre de Virasoro
N = 2 tordue, Ann. Inst. Four. 58 (2008), 733--754.



\bibitem[Io10]{iohara} ---, \emph{Unitarizable highest weight modules of the $N=2$ super Virasoro algebras: untwisted sectors}, Lett. Math. Phys. 91 (2010), no. 3, 289--305.

\bibitem[JaZh88]{jarzh}
Jarvis, P. D., Zhang, R. B.,
\emph{Unitary Sugawara constructions for affine superalgebras},
Phys. Lett. B 215 (1988), no. 4, 695--700. 


\bibitem[JaKa]{jakobsen2} Jakobsen, H. P., Kac, V.,
\emph{A new class of unitarizable highest weight representations of infinite-dimensional Lie algebras. II},
J. Funct. Anal. 82 (1989), no. 1, 69--90. 




\bibitem[KaTo]{kactodor} Kac, V. G., Todorov, I. T., 
\emph{Superconformal current algebras and their unitary representation}, Comm. Math. Phys. 102 (1985), 
337--347.


\bibitem[KoB]{kostant} Kostant, B., \emph{Graded manifolds, graded Lie theory, and prequantization}, Differential geometrical methods in mathematical physics (Proc. Sympos., Univ. Bonn, Bonn, 1975), pp. 177--306. Lecture Notes in Math. 570, Springer, Berlin, 1977.


\bibitem[KoJ]{koszul} Koszul, J.-L. \emph{Graded manifolds and graded Lie algebras}, Proceedings of the international meeting on geometry and physics (Florence, 1982), 71--84, Pitagora, Bologna, 1983.




\bibitem[KrMi97]{krieglmichor} Kriegl A., Michor, P. W.,
\emph{The convenient setting of global analysis},
American Mathematical Society, 1997.

\bibitem[MNS12]{menesa} Merigon, S., Neeb, K.--H., Salmasian, S., 
\emph{Categories of unitary representations of Banach--Lie supergroups and restriction functors}, to appear in Pacific Journal of Mathematics 257 (2012), no. 2, p.431--470.


\bibitem[Mi84]{milnor}
Milnor, J., \emph{Remarks on infinite-dimensional Lie groups}, in: B. DeWitt, R. Stora (Eds.), Relativit\' e, groupes et
topologie II, Les Houches, 1983, North-Holland, Amsterdam, 1984, p. 1007--1057.

\bibitem[Ne00]{neebbook} Neeb, K.--H., \emph{Holomorphy and convexity in Lie theory}, Walter de Gruyter, 2000. 


\bibitem[Ne06]{neebjpn} ---, \emph{Towards a Lie theory of locally convex groups}, Japanese Journal of Mathematics,  1 (2006), p.291--468.



\bibitem[Ne10]{neebdiff} ---,  \emph{On differentiable vectors for representations of infinite-dimensional Lie groups},
J. Funct. Anal., 259 (2010), no. 11, 2814--2855. 


\bibitem[Ne11]{neebanalytic} ---,  \emph{On analytic vectors for unitary representations of infinite-dimensional Lie groups}, Ann. Inst. Fourier (Grenoble), 
61 (2011), no. 5, p. 1839--1874.

\bibitem[NeSa12]{nsfreshetsuper} Neeb, K--H., Salmasian, H., 
\emph{Differentiable vectors and unitary representations of Fr\' echet--Lie supergroups}, preprint.


\bibitem[Sa09]{sachse} Sachse, C., \emph{Global analytic approach to Teichm\"uller spaces,} Ph.D., Thesis, Universi\"at Leipzig, 2007.

\end{thebibliography}
\end{document}